\documentclass{article}
\usepackage{arxiv}
\usepackage{mathtools} 
\usepackage[utf8]{inputenc}

\usepackage{setspace}
\usepackage{empheq} 

\usepackage{amssymb,amsmath,amsthm,eucal}

\usepackage[english]{babel}

\usepackage{times,cite}
\usepackage[T1]{fontenc}
\usepackage[utf8]{inputenc}

\numberwithin{equation}{section}

\usepackage{graphicx}
\usepackage{amsfonts}
\usepackage{booktabs}
\usepackage{multirow}

\usepackage{float}
\usepackage{subcaption}

\usepackage{hyperref}
\usepackage{algorithm}
\usepackage{algpseudocode}
\usepackage{xcolor}
\usepackage{color}

\newtheorem{lemma}{Lemma}
\newtheorem{remark}{Remark}
\newtheorem{theorem}{Theorem}

\title{Exponential Consensus through Z-Control in High-Order Multi-Agent Systems}

\author{
 Angela Monti \\
  Istituto per le Applicazioni del Calcolo \lq \lq M. Picone\rq \rq \\
  National Research Council (CNR)\\
  via G. Amendola 122/D, Bari, Italy\\
  \texttt{angela.monti@cnr.it} 
  \And
  Fasma Diele \\
  Istituto per le Applicazioni del Calcolo \lq \lq M. Picone\rq \rq \\
  National Research Council (CNR)\\
  via G. Amendola 122/D, Bari, Italy\\
  \texttt{fasma.diele@cnr.it}
}

\begin{document}

\maketitle
\begin{abstract}

In this work, we introduce a Z-control strategy  for multi-agent systems of arbitrary order, aimed at driving the agents toward consensus in the highest-order observable state. The proposed framework supports both direct and indirect control schemes, making it applicable in scenarios where high-order derivatives such as acceleration cannot be directly manipulated. Theoretical analysis ensures exponential convergence while preserving the average dynamics, and a hierarchy of control laws is derived accordingly. Numerical experiments up to third-order models, including opinion dynamics and Cucker-Smale flocking systems, demonstrate the robustness and flexibility of Z-control under varying interaction regimes and control intensities.
\end{abstract}
\section{Introduction}

The study of consensus in multi-agent systems is a central topic in control theory and complex systems \cite{olfati2004consensus,ren2008distributed}, with applications ranging from robotics and biological swarms \cite{olfati2006flocking,vicsek1995novel} to social behaviour \cite{proskurnikov2018tutorial} and epidemiology \cite{nowzari2016analysis}. Classical consensus models have traditionally focused on first- or second-order dynamics \cite{he2011consensus, ren2006high}
where agents align their state or velocity through local interaction rules. However, many real-world systems, such as autonomous vehicles, collective motion in animals, or socio-dynamic phenomena, exhibit higher-order behaviors, necessitating more general and flexible control strategies.

The Z-control approach is an error-driven feedback design originally developed within neural-dynamic methods for solving dynamic and optimization problems \cite{GuoZhang2014,LiaoZhang2014,ZhangLi2009,ZhangYi2011}. In recent years it has moved beyond engineering, finding applications in the life sciences. In population dynamics it has been used on classical and generalized Lotka-Volterra systems to promote ecological coexistence \cite{ZhangEtAl2016,MBS2016}; in epidemic modeling it has been applied to simple susceptible-infected frameworks to regulate disease spread \cite{Samanta2018,SenapatiEtAl2020} and to suppress backward (subthreshold endemic) scenarios \cite{LacitignolaDiele2019}; and in eco-epidemiology it has been employed to prevent chaotic oscillations in predator-prey systems with infection in the prey \cite{AlzahraniEtAl2018,MandalEtAl2021}; and, more recently, in sustainability, it has been applied to referral marketing for the circular economy \cite{LMcircular2025}.

Within these settings, Z-control may be implemented either as \emph{direct} control, where external actions are directly  applied to  target populations, or as \emph{indirect} control, intervening on one population to steer another. A key advantage is its ability to reshape the uncontrolled dynamics so that trajectories converge to a prescribed target equilibrium, which becomes the unique attractor of the Z-controlled model \cite{MBS2016}. Effective use, however, requires careful calibration of the feedback strength to balance convergence speed with model constraints \cite{LacitignolaDiele2019}.

In this work, we propose a novel, hierarchical Z-control methodology, designed to stabilize consensus in multi-agent systems of arbitrary order $k$. Unlike classical approaches that assume only direct control over the highest-order state variable, Z-control also supports indirect control, whereby consensus is achieved at the top-level state through intervention on lower-order dynamics (e.g., velocity or position). This is particularly valuable in scenarios where direct manipulation of the  highest order observable state is not physically feasible \cite{nakai2025guidance,lozano2016phototaxis}.

Several alternative strategies have been proposed in the literature for enforcing consensus. Among the most notable are optimal control methods, such as those studied in~\cite{bailo2018optimal, monti2025hierarchicalclusteringdimensionalreduction} for the Cucker-Smale model or opinion dynamics. Their approach casts the consensus problem as a finite-horizon dynamic optimization, balancing control energy against deviation from consensus. While powerful and theoretically grounded, these methods often suffer from scalability issues typically faced numerically via iterative schemes, making real-time implementation challenging in large-scale or fast-evolving systems.

In contrast, the Z-control framework offers a fully analytical and scalable approach. By recursively constructing a chain of error dynamics, Z-control transforms the consensus goal into a sequence of stabilizing conditions imposed on lower-order state variables. This leads to a family of feedback laws that can be tailored to act either directly on the highest-order derivative (direct Z-control) or indirectly on lower order derivatives as  position and velocity-level dynamics (indirect Z-control), while preserving the system's conserved quantities induced by the structural properties of the interaction strength matrix. In this paper we focus on interaction strength matrices that are \emph{weight-balanced} (see, e.g., \cite{gharesifard2012ejc}):
this assumption is a mild relaxation of symmetry that preserves several key spectral and averaging properties of undirected Laplacians (see, e.g., \cite{chung1997spectral}).

This work extends Z-control to general $k$-th order multi-agent systems, building on earlier studies \cite{LacitignolaDiele2019,MBS2016,lacitignola2021using}. Those contributions laid the groundwork for the theoretical generalization developed here, showing that Z-control can steer such systems toward prescribed consensus via indirect, and often more practical, interventions \cite{lozano2016phototaxis,nakai2025guidance}.

To illustrate the power and flexibility of our approach, we apply Z-control to canonical multi-agent models: a first-order opinion dynamics system inspired by the Hegselmann-Krause framework, the second-order Cucker-Smale flocking model and its generalized third-order version with both direct and indirect control strategies. Through analytical derivation and numerical simulations, we show that Z-control ensures exponential convergence to consensus and allows fine-tuned trade-offs between control effort and convergence speed, controlled by the design parameter $\lambda$. We show that the method remains effective under weak or highly localized interaction regimes, where uncontrolled systems typically fail to self-organize.


The rest of the paper is organized as follows: Section~\ref{sec:model} formalizes the $k$-th order consensus dynamics. Section~\ref{sec:zcontrol} introduces Z-control and its recursive design. Section~\ref{sec:applications} presents numerical results on opinion dynamics and flocking. Section~\ref{sec:stacked-solve} presents the least-squares formulation and solver for indirect Z-control, and Section~\ref{sec:conclusion} concludes with future directions.

\section{General $k$-th Order Consensus Models}\label{sec:model}

We consider a general class of $k$-th order consensus models in which each agent
$i\in\{1,\dots,N\}$ is characterized by the sequence of state variables
\[
x_i^{(1)}(t),\, x_i^{(2)}(t),\, \dots,\, x_i^{(k)}(t) \in \mathbb{R}^d,
\]
corresponding to position, velocity, acceleration, and higher-order derivatives of the position
$x_i^{(1)}(t) := x_i(t)$.

The dynamics evolve as
\begin{equation} \label{k-order_system}
\begin{cases}
\dot{x}_i^{(1)}(t) = x_i^{(2)}(t), \\
\dot{x}_i^{(2)}(t) = x_i^{(3)}(t), \\
\ \ \vdots \\
\dot{x}_i^{(k-1)}(t) = x_i^{(k)}(t), \\
\dot{x}_i^{(k)}(t) = \displaystyle\sum_{j=1}^N a_{ij}\!\big(X^{(1)}(t)\big)\,\big(x_j^{(k)}(t) - x_i^{(k)}(t)\big),
\end{cases}
\qquad i = 1,\dots,N,
\end{equation}
where $a_{ij}\!\big(X^{(1)}(t)\big)\ge 0$ denotes the interaction strength between agents $i$ and $j$, possibly depending on the position matrix 
$X^{(1)}(t)\in\mathbb{R}^{N\times d}$ obtained by placing each $d$ dimensional vector $\big(x_i^{(1)}(t)\big)^\top$ in the $i$-th row, for $i=1,\dots N$.

In this paper, we restrict our attention to the case in which $A=(a_{ij})$ is \emph{weight-balanced} (see, e.g., \cite{gharesifard2012ejc}), namely
\[
\sum_{j=1}^N a_{ij} \;=\; \sum_{j=1}^N a_{ji} \qquad \text{for all } i.
\]
For Laplacian-type matrices, weight balance is equivalent to simultaneously having zero row sums  ($A \,  1_N \,= 0_N$) and zero column sums ($1_N^\top \, A \,  \,= 0_N$),  where $1_N$ and $0_N$ denote the $N$-dimensional column vectors of ones and zeros, respectively.

In networked systems, weight-balanced matrices play a central role: they guarantee invariance of the arithmetic average for consensus dynamics and allow many undirected results to extend to directed settings. This is exploited in continuous-time and discrete-time average consensus \cite{xiao2004fast,olfati2004consensus} and in continuous-time distributed convex optimization over directed graphs, where weight balance enables primal dynamics with convergence guarantees \cite{gharesifard2012ejc}. There is also an active literature on generating weight-balanced (or even doubly stochastic) weights in a distributed manner \cite{gharesifard2012ejc}.

\subsection{Consensus and invariance of average}

We say the $k$-th order consensus is achieved if:
\[
\lim_{t \to \infty} \left\| x_i^{(k)}(t) - \bar{x}^{(k)}(t) \right\| = 0, \quad \forall i,
\]
where \( \bar{x}^{(k)}(t) := \frac{1}{N} \sum_{i=1}^N x_i^{(k)}(t) \) is the average of the highest-order state.
\begin{lemma}[Invariance of the average under weight-balanced interactions]
Let $x_i^{(k)}(t)\in\mathbb{R}^d$ evolve according to \eqref{k-order_system}

where the interaction matrix $A=(a_{ij})$ is weight-balanced. Then the population average
\[
\bar{x}^{(k)}(t)\;:=\;\frac{1}{N}\sum_{i=1}^N x_i^{(k)}(t)
\]
is preserved.
\end{lemma}

\begin{proof} 

By definition, 
\begin{align*}
\dot{\bar{x}}^{(k)}(t)
&= \frac{1}{N}\sum_{i=1}^{N}\dot{x}_i^{(k)}(t)
 = \frac{1}{N}\sum_{i,j=1}^{N} a_{ij}\bigl(x_j^{(k)}(t)-x_i^{(k)}(t)\bigr)\\
&= \frac{1}{N}\sum_{i,j=1}^{N} a_{ij}\,x_j^{(k)}(t)
 - \frac{1}{N}\sum_{i,j=1}^{N} a_{ij}\,x_i^{(k)}(t)
 = \frac{1}{N}\sum_{j=1}^{N} x_j^{(k)}(t)\sum_{i=1}^{N} a_{ij}
 - \frac{1}{N}\sum_{i=1}^{N} x_i^{(k)}(t)\sum_{j=1}^{N} a_{ij}\\
&= \frac{1}{N}\sum_{j=1}^{N} x_j^{(k)}(t)\sum_{i=1}^{N} a_{ij}
 - \frac{1}{N}\sum_{i=1}^{N} x_i^{(k)}(t)\sum_{j=1}^{N} a_{ji}
 = \frac{1}{N}\sum_{i,j=1}^{N} a_{ij}\,x_j^{(k)}(t)
 - \frac{1}{N}\sum_{i,j=1}^{N} a_{ji}\,x_i^{(k)}(t)\\
&= \frac{1}{N}\sum_{i,j=1}^{N} a_{ij}\,x_j^{(k)}(t)
 - \frac{1}{N}\sum_{i,j=1}^{N} a_{ij}\,x_j^{(k)}(t)
 = 0_d.
\end{align*}

\end{proof}
To quantify the deviation from consensus, we introduce the consensus parameter

\begin{equation}
    \label{eq:consensus}
    \Gamma(t) = \frac{1}{N^2}\sum_{i=1}^{N}
\|x_i^{(k)}(t) - \bar{x}^{(k)}(t)\|^2,
\end{equation}
so that a solution to \eqref{k-order_system} tends to a consensus configuration if and only if $\Gamma(t) \to 0$ as $t \to \infty$.

\section{Z-Controlled \(k\)-th Order Multi-Agent Models}\label{sec:zcontrol}

We consider the general system of order \(k\) as defined in (\ref{k-order_system}). The objective is to drive the system to consensus in the highest observable state, that is,
\[
\lim_{t \to \infty} \|x_i^{(k)}(t) - \bar{x}^{(k)}(t)\| = 0, \quad \text{with } \bar{x}^{(k)}(t) := \frac{1}{N} \sum_{i=1}^N x_i^{(k)}(t),
\]
by designing the control input \( u_i(t) \) such that the error
\begin{equation}
    \label{errore}
e_i^{[1]}(t) := x_i^{(k)}(t) - \bar{x}^{(k)}(0)
\end{equation}
decays exponentially with rate \(\lambda\), while preserving the average $\bar{x}^{(k)}$ i.e.  \( \bar{x}^{(k)}(t)= \bar{x}^{(k)}(0)\).

\subsection{Direct and Indirect Z-Control}

There are two main paradigms for designing the control \( u_i(t) \):
\paragraph{Direct Z-Control.}
The control is applied directly to the dynamics of \(x_i^{(k)}\). The system is

\[
\dot{x}_i^{(1)}(t) = x_i^{(2)}(t),\quad
\dot{x}_i^{(2)}(t) = x_i^{(3)}(t),\ \dots,\ 
\dot{x}_i^{(k-1)}(t) = x_i^{(k)}(t),
\]
\begin{equation}\label{eq:ultima}
\dot{x}_i^{(k)}(t) = \sum_{j=1}^N a_{ij}\!\big(X^{(1)}(t)\big)\,\big(x_j^{(k)}(t) - x_i^{(k)}(t)\big) + u_i(t),
\end{equation}
and the control is designed by stabilizing the error between \(x_i^{(k)}\)(t) and its average \( \bar{x}^{(k)}(t)\).

Before applying the Z-control technique, we state a preliminary result for any control acting directly on \({x}_i^{(k)}\).

\begin{theorem}\label{the:sumzerocontrol}
If \(A\) is weight-balanced, then controls applied to the dynamics of \(x_i^{(k)}\) preserve the average \(\bar{x}^{(k)}(t)\) (as in the uncontrolled case)  if and only if the inputs are zero-sum at all times, i.e., \(\sum_{i=1}^N u_i(t)=0_d\) for all \(t\).
\end{theorem}

\begin{proof}
Let us consider the Z-controlled dynamics \eqref{eq:ultima} and evaluate the time derivative of the average
\begin{align*}
\dot{\bar{x}}^{(k)}(t) &= \frac{1}{N} \sum_{i=1}^{N}\dot{x}_i^{(k)}(t) = \frac{1}{N}\sum_{i=1}^{N} \Bigg( \sum_{j=1}^{N} a_{ij}\left(x_j^{(k)}(t)-x_i^{(k)}(t)\right)+u_i(t)\Bigg) \\
&= \frac{1}{N}\sum_{i,j=1}^{N} a_{ij}\left(x_j^{(k)}(t)-x_i^{(k)}(t)\right) + \frac{1}{N} \sum_{i=1}^{N} u_i(t)\, = \frac{1}{N}\, \sum_{i=1}^{N} u_i(t), 
\end{align*}
where $\sum_{i,j=1}^{N} a_{ij}\left(x_j^{(k)}(t)-x_i^{(k)}(t)\right) = 0_d$ follows from the weight-balance property of $A$. As a consequence, 
\begin{equation*}
 \dot{\bar{x}}^{(k)}(t) = 0_d \Longleftrightarrow \sum_{i=1}^{N} u_i(t) = 0_d \quad \text{for all } t,  
\end{equation*}
thus the average $\bar{x}^{(k)}(t)$ is preserved if and only if the inputs $u_i(t)$ have zero sum at all times. 
\end{proof}

Let us now focus on the Z-control applied to \(x_i^{(k)}\).
\begin{theorem}
    Suppose $A$ weight-balanced and set $\lambda >0$. Consider the control law 
\begin{equation}\label{eq:direct-control}
u_i(t)
= -\lambda\big(x_i^{(k)}(t)-\bar{x}^{(k)}(t)\big)
 - \sum_{j=1}^N a_{ij}\!\big(X^{(1)}(t)\big)\,\big(x_j^{(k)}(t)-x_i^{(k)}(t)\big), \qquad  i=1,\dots, N. 
\end{equation} The solution of the controlled dynamics 
(\ref{eq:ultima}) with controls defined in (\ref{eq:direct-control}) preserves the average  \(\bar{x}^{(k)}(t)=\bar{x}^{(k)}(0)\). Moreover, the solution  of the controlled dynamics 
(\ref{eq:ultima}) with controls defined in (\ref{eq:direct-control}) satisfies the Z-control design equation 
\begin{equation*}
\label{eq:first-order}
\dot e_i^{[1]} = -\lambda\, e_i^{[1]}.
\end{equation*}
i.e. the error 
$e_i^{[1]}(t)$, defined in \eqref{errore}, 
decays exponentially with rate \(\lambda\).
\end{theorem}



\begin{proof}
 Let us consider the controls defined by \eqref{eq:direct-control}. We have
\begin{align*}
    \sum_{i=1}^{N} u_i(t) &= \sum_{i=1}^{N} \Bigg(-\lambda\left(x_i^{(k)}(t)-\bar{x}^{(k)}(t)\right)-\sum_{j=1}^{N} a_{ij}(X^{(1)})\left(x_j^{(k)}(t) - x_i^{(k)}(t)\right)\Bigg)\\
    &= -\lambda \sum_{i=1}^{N} x_i^{(k)}(t)+ \lambda \sum_{i=1}^{N}\bar{x}^{(k)}(t)-\sum_{i,j=1}^{N} a_{ij}(X^{(1)})\left(x_j^{(k)}(t) - x_i^{(k)}(t)\right)\\
    &= -\lambda\,N\,\bar{x}^{(k)}(t) +\lambda\,N\,\bar{x}^{(k)}(t)= 0_d,
\end{align*}
because the double sum vanishes on weight-balanced matrices. 
From Theorem \ref{the:sumzerocontrol} the result follows.
Moreover, using the conservation property of the average of $x_i^{(k)}$, evaluate
\begin{align*}
\dot e_i^{[1]}(t)
&= \frac{d}{dt}\big(x_i^{(k)}(t)-\bar{x}^{(k)}(0)\big)
 = \dot x_i^{(k)}(t)
\, =\,  \sum_{j=1}^N a_{ij}\!\big(X^{(1)}(t)\big)\big(x_j^{(k)}(t)-x_i^{(k)}(t)\big) + u_i(t)\\
&= \sum_{j=1}^N a_{ij}\!\big(X^{(1)}(t)\big)\big(x_j^{(k)}(t)-x_i^{(k)}(t)\big)
   - \lambda\big(x_i^{(k)}(t)-\bar{x}^{(k)}(t)\big)
   - \sum_{j=1}^N a_{ij}\!\big(X^{(1)}(t)\big)\big(x_j^{(k)}(t)-x_i^{(k)}(t)\big)\\
&= -\lambda\big(x_i^{(k)}(t)-\bar{x}^{(k)}(t)\big)
 = -\lambda\big(x_i^{(k)}(t)-\bar{x}^{(k)}(0)\big)
 = -\lambda\,e_i^{[1]}(t).
\end{align*}
\end{proof}

\paragraph{Indirect Z-control.}
Indirect Z-control is particularly valuable when direct actuation of higher-order states (e.g., accelerations or jerks) is not feasible. The idea is to apply the control input on a lower-order state \(x_i^{(r)}\), with \(r\in\{1,\dots,k-1\}\), while still driving consensus on the highest-order state \(x_i^{(k)}\). The dynamics are

\begin{subequations}
\label{eq:indirect-kr}
\begin{empheq}[left=\left\{\;,right=\right.]{align}
\dot{x}_i^{(1)}(t) &= x_i^{(2)}(t), 
\notag \\
& \vdots \nonumber \\
\dot{x}_i^{(r)}(t) &= x_i^{(r+1)}(t) + u_i(t), \label{eq:indirect-r} \\
& \vdots \nonumber \\
\dot{x}_i^{(k)}(t) &= \sum_{j=1}^N a_{ij}\!\big(X^{(1)}(t)\big)\,\big(x_j^{(k)}(t)-x_i^{(k)}(t)\big). \label{eq:indirect-k}
\end{empheq} 
\end{subequations}

\begin{remark}[Conservation of the average of \(x^{(k)}\)]
Since the control \(u_i(t)\) enters only in \eqref{eq:indirect-r} with \(r<k\), it does \emph{not} appear in \eqref{eq:indirect-k}. Hence, due to the weight-balance property of $A$, the average $\bar{x}^{(k)}(t)$
is conserved as in the uncontrolled case.
\end{remark}

To achieve exponential consensus of \( x_i^{(k)} \), we define a recursive sequence of error terms:
\[
e_i^{[1]}(t) := x_i^{(k)}(t) - \bar{x}^{(k)}(t), \qquad 
e_i^{[j]}(t) := \frac{d}{dt} e_i^{[j-1]}(t) + \lambda \, e_i^{[j-1]}(t), \quad j = 2, \dots, r,
\]
and design the control to enforce:
\[
e_i^{[r]}(t) = 0_d.
\]

As established in Proposition~3.1 of~\cite{lacitignola2021using}, this condition is equivalent to requiring that the first-order error \( e_i^{[1]} \) satisfies a linear differential equation of order \( r \):
\[
\sum_{j=0}^{r} \binom{r}{j} \lambda^j \frac{d^{r-j}}{dt^{r-j}} e_i^{[1]}(t) = 0_d.
\]

The resulting control input \( u_i(t) \), acting through the dynamics of \( x_i^{(r)} \), implements a feedback law of differential order \( r+1 \), while ensuring exponential convergence of the highest-order state \( x_i^{(k)} \) to consensus.

This framework naturally gives rise to a hierarchy of control strategies: \begin{itemize} \item Indirect Z-control on \( x_i^{(k)} \) via \( x_i^{(1)} \), using an order-\( 2 \) error formula; \item Indirect Z-control on \( x_i^{(k)} \) via \( x_i^{(2)} \), using an order-\( 3 \) error formula; \item \dots \item Indirect Z-control on \( x_i^{(k)} \) via \( x_i^{(k-1)} \), using an order-\( k \) error formula. \end{itemize}

\section{Applications to Opinion Dynamics and Flocking Models}\label{sec:applications}
\subsection{First-Order Opinion Dynamics Model}

To illustrate the general framework for implementing Z-control, we consider the concrete example of first-order ABMs describing opinion dynamics \cite{motsch2014heterophilious}.
In these models, $N$ agents interact
with each other according to the first-order system
\begin{equation}
\label{abm}
\dot{x}_i(t) = \sum_{j=1}^N a_{ij}(x(t))\bigl(x_j(t) - x_i(t)\bigr),
\qquad i=1,\dots,N,
\end{equation}
where the interaction coefficients $a_{ij}(x)$ depend on the current
configuration $x=(x_1,\dots,x_N)$.

We start from a symmetric  interaction kernel $\phi:[0,\infty)\to(0,1]$. Here, the function $\phi$ models the strength of interactions between agents, 
each having a vector of opinions represented by the state 
$\mathbf{x}_i \in \mathbb{R}^d$. 
The interaction depends only on the relative distance between agents' opinions,
measured using the Euclidean norm. Set
\begin{equation*}
\phi_{ij}(x) :=
\phi(\|x_i-x_j\|).
\end{equation*}
To introduce a directed,  weight-balanced, interaction pattern, we fix
a constant $\beta\in(0,1)$ and define
\begin{equation}
\label{eq:epsilon1st}
\varepsilon(x) := \beta \min_{i=2,\dots,N} \left[ \phi_{i,i-1}(x), \phi_{1,N}(x)\right]. 
\end{equation}
We also consider the constant matrix $S=(s_{ij})_{i,j=1}^N$ given by
\begin{equation*}
s_{1,N}:=\,-1, \quad s_{N,1}:=\,1, \quad   
s_{ij} :=
\begin{cases}
+1, & j = i+1, \quad i<N,\\[2pt]
-1, & j = i-1, \quad i>1,\\[2pt]
0,  & \text{otherwise},
\end{cases}
\end{equation*}
which is skew-symmetric ($S^\top=-S$) and satisfies
$\sum_{j=1}^N s_{ij}=0=\sum_{j=1}^N s_{ji}$ for all $i$.

The interaction weights are then defined by
\[
a_{ij}(x) :=
\phi_{ij}(x) + \varepsilon(x)\,s_{ij}.
\]
By construction $a_{ij}(x)\ge 0$ for all $i\neq j$, since only the entries
corresponding to edges of the cycle $(i-1,i,i+1)$ are modified and
$\varepsilon(x)\le \phi_{i,i-1}(x)$.
Moreover, for every $i$ we have
\begin{equation*}
\sum_{j=1}^N a_{ij}(x)
= \sum_{j=1}^N \phi_{ij}(x) + \varepsilon(x)\sum_{j=1}^N s_{ij}
= \sum_{j=1}^N \phi_{ij} = \sum_{j=1}^N \phi_{ji}(x) + \varepsilon(x)\sum_{j=1}^N s_{ji}(x) = \sum_{j=1}^N a_{ji}(x)
\end{equation*}

for all $i$, where we used the symmetry $\phi_{ij}=\phi_{ji}$ and the zero row/column sums
of $S$. Hence
the adjacency matrix $A(x)=(a_{ij}(x))$ is weight-balanced.
In general, for $\varepsilon(x)>0$ the matrix $A(x)$ is not symmetric, since
$a_{i,i+1}(x) - a_{i+1,i}(x) = 2\varepsilon(x)\neq 0$, and thus it defines a
directed, weight-balanced interaction structure.

\noindent As the model is of first order, only direct Z-control can be applied. 
From (\ref{eq:direct-control}), controls that guarantee asymptotic convergence toward consensus are given by 
$$
u_i= -\lambda (x_i\,-\, \bar x) - \sum_{j=1}^{N} \left( \phi_{ij}(x) + \varepsilon(x)\,k_{ij} \right)\,(x_j-x_i).
$$
\paragraph{Numerical example.}To demonstrate the effectiveness of the Z-control framework, we use the  smoothed generalized Hegselmann-Krause influence functions on infinite support:
\[
\phi(x) := \frac{1 - \text{sig}(\alpha(x - 1))}{1 - \text{sig}(-\alpha)}, \quad \text{sig}(y) = \frac{1}{1 + e^{-y}}
\]
where \( \alpha \) is a parameter to control the smoothness of the curve. 
 We consider a system of $N = 10$ agents evolving in $\mathbb{R}^2$ according to the interaction dynamics \eqref{abm}, with $\beta = 0.8$ in \eqref{eq:epsilon1st}. We simulate the uncontrolled system using the smoothed interaction function $\phi(x)$ for the four representative values $\alpha = 0.1$, $\alpha = 1.6$, $\alpha = 5$, and $\alpha = 300$. Figure~\ref{fig:alpha} (on the left) shows the evolution of the agents' opinions under the smoothed generalized Hegselmann-Krause dynamics for different values of the parameter $\alpha$. For small values of $\alpha$ (e.g., $\alpha = 0.1$), the interaction kernel decays slowly, so agents influence each other over long distances. This leads to an almost global coupling and rapid convergence to consensus. As $\alpha$ increases (e.g., $\alpha = 1.6$), the interaction becomes more localized: opinions still tend toward a common value, but the process is slower and less uniform. For intermediate values such as $\alpha = 5$, the influence range is further reduced, and multiple opinion clusters emerge; some of these may merge over time, producing a few dominant groups rather than full consensus. Finally, for very large $\alpha$ (e.g., $\alpha = 300$), interactions are restricted to very close neighbors, and the initial fragmentation persists indefinitely, resulting in stable, distinct clusters.

\begin{figure}[htbp]
    \centering
    \includegraphics[width=0.33\textwidth]
    {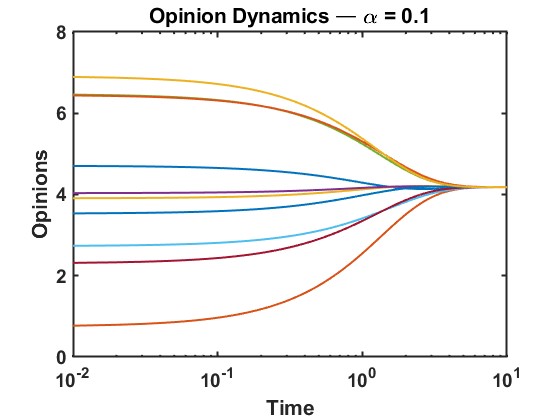} 
\includegraphics[width=0.33\textwidth]{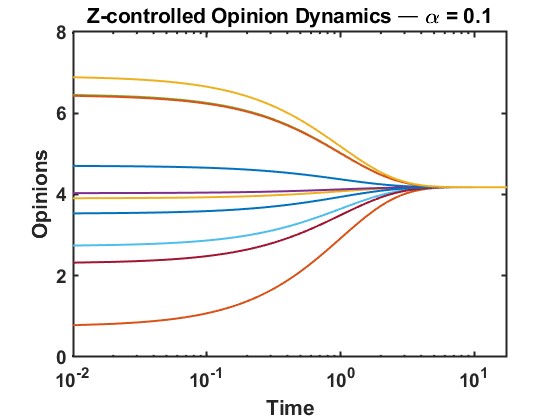}
\includegraphics[width=0.33\textwidth]{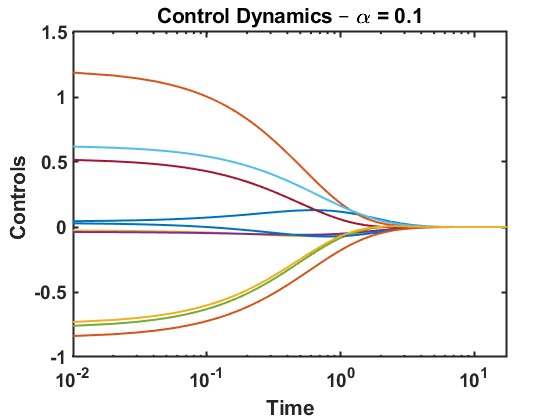}
\includegraphics[width=0.33\textwidth]   {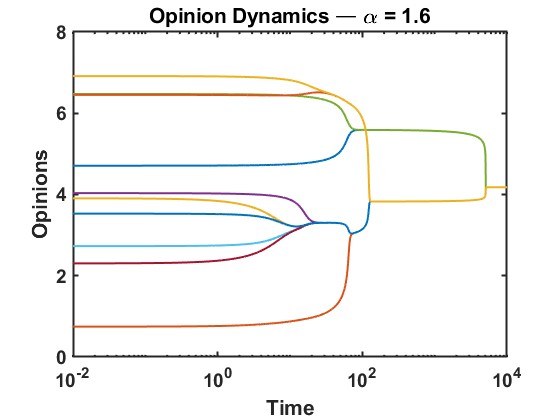}
\includegraphics[width=0.33\textwidth]{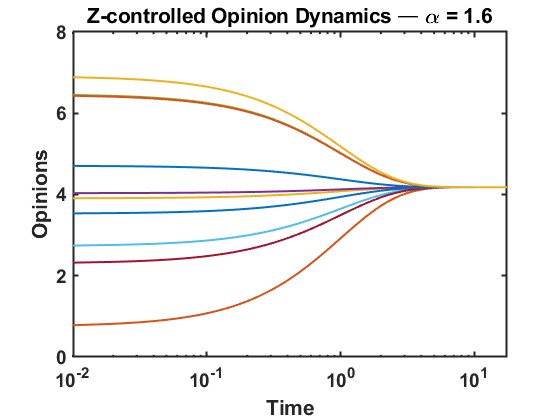}
\includegraphics[width=0.33\textwidth]{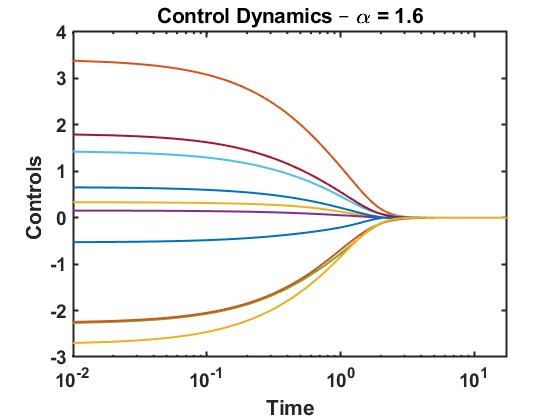}
\includegraphics[width=0.33\textwidth]{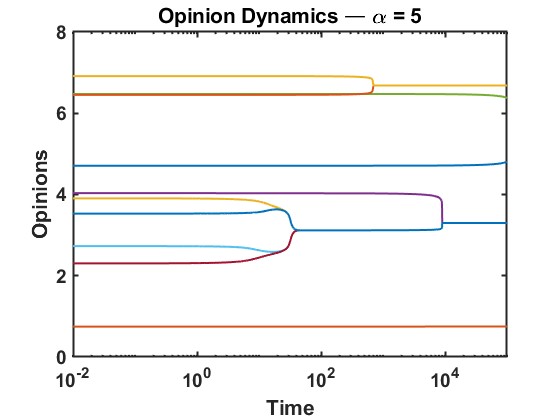}
\includegraphics[width=0.33\textwidth]{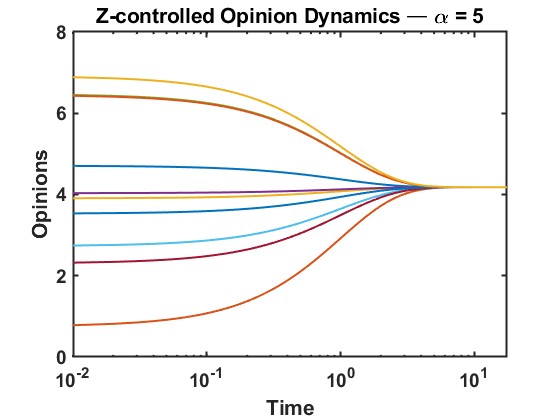}
\includegraphics[width=0.33\textwidth]{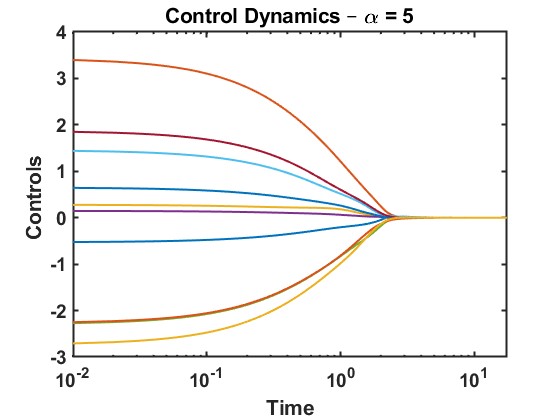}
\includegraphics[width=0.33\textwidth]{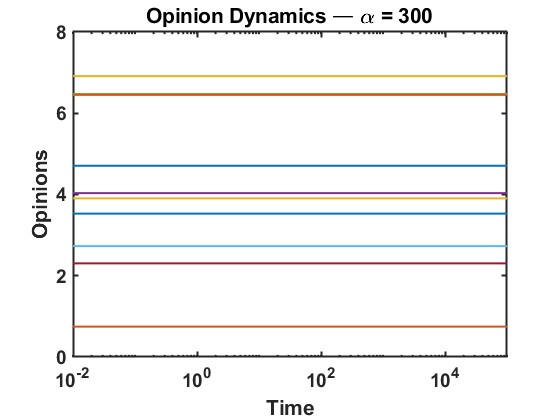} \includegraphics[width=0.33\textwidth]{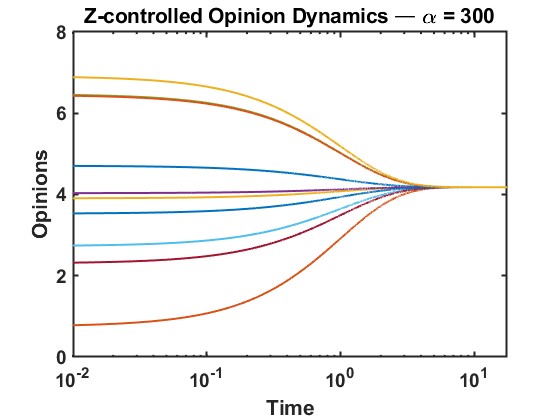}
\includegraphics[width=0.33\textwidth]{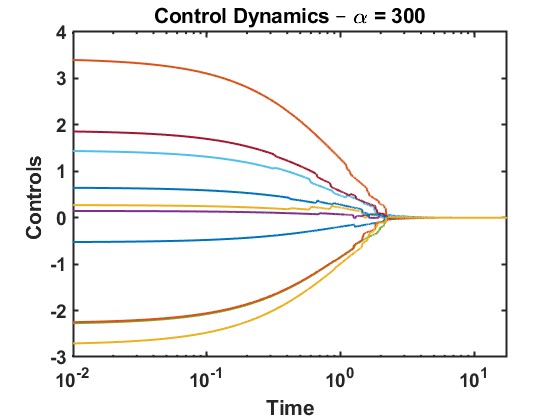}
    \caption{Opinion dynamics for four different values of the smoothness parameter $\alpha$. Left panels: uncontrolled dynamics. As $\alpha$ increases, the interaction function becomes more localized, transitioning from global consensus to the persistence of multiple opinion clusters. Center panels: Z-controlled dynamics converging to consensus. The controlled dynamics are similar for different values of $\alpha$: they converge to consensus with an error that decreases at the same rate, governed by $\lambda = 1$. Right panels: control signals required to steer the system toward consensus. For $\alpha = 0.1$  the controls remain around $\pm0.2$, while for $\alpha = 300$, significantly larger control values are needed to force the dynamics towards consensus.}
    \label{fig:alpha}
\end{figure}

To assess the effectiveness of the proposed strategy in enforcing consensus, we fix $\lambda=1$ and apply Z-control across different interaction regimes. The center panel of Figure~\ref{fig:alpha} shows the trajectories of the controlled opinions, while the right panel 
reports the time evolution of the control inputs 
. As expected, the controlled dynamics are insensitive to $\alpha$, since the convergence rate is determined solely by $\lambda$. However, the control effort increases markedly for larger values of $\alpha$, reflecting the greater input required when interactions are weaker.
To explore how we can control the velocity of the proposed Z-control strategy for consensus, we apply it under different values of the $\lambda$ parameter. In Figure \ref{fig:zcontrol_consensus}, we report the evolution of the consensus parameter $\Gamma(t)$, defined in \eqref{eq:consensus}, for different values of $\lambda$ under Z-control. All trajectories exhibit an exponential decay, confirming that consensus is achieved in every case. However, the convergence rate strongly depends on $\lambda$: larger values of $\lambda$ lead to a significantly faster decrease of $\Gamma(t)$, while smaller values result in a much slower approach to consensus. 

\begin{figure}[htbp]
    \centering
    \includegraphics[width=0.4\linewidth]{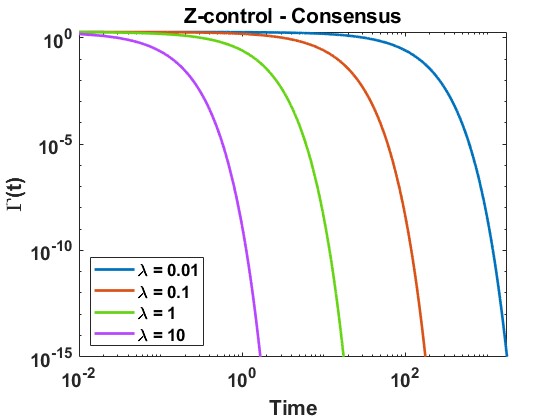}
    \caption{Consensus parameter of the controlled opinion dynamics for $\alpha = 1.6$ and increasing values of the parameter $\lambda$. Larger values of $\lambda$ induce a faster convergence. Note that achieving consensus requires setting the product $\lambda \, T$ to approximately $10$.}
    \label{fig:zcontrol_consensus}
\end{figure}



\subsection{Second-Order Cucker-Smale Flocking Model}
In this Section, we extend the Z-control strategy to a second-order multi-agent system governed by the Cucker-Smale model, a framework widely used to describe flocking dynamics. Unlike first-order models, where agents update their state directly based on interaction rules, second-order models incorporate both position and velocity, providing a more realistic representation of collective motion in natural and artificial systems.

The model is described by the following system of equations:
\begin{equation}\label{eq:CS}
    \begin{cases}
    \dot{x}_i = v_i, \\[6pt]
    \dot{v}_i = \sum_{j=1}^{N} a_{ij}(x)(v_j - v_i), 
    \end{cases}
    \quad a_{ij}(x) = \frac{K}{N \, (1 + \|x_i - x_j\|^2)^\beta}, \quad i,j  = 1, \dots, N, i \neq j
\end{equation}
where \( K > 0 \) denotes the interaction strength and \( \beta \geq 0 \) determines the rate at which interaction decays with distance.

It is well established that for \( 0 < \beta \leq \frac{1}{2} \), the system exhibits unconditional convergence to velocity consensus \cite{CuckerSmale2007,CarrilloEtAl2010b}. For \( \beta > \frac{1}{2} \), consensus emergence becomes conditional and depends on the initial configuration of the agents \cite{HaLiu2009,HaHaKim2010}.

To show the effect of the parameter $\beta$ on the uncontrolled model, we set $N = 10$ agents that evolve in $\mathbb{R}^2$ according to the interaction dynamics \eqref{eq:CS}, and we simulate the system using $K=1$ and the representative values $\beta = 0.1$ and $\beta = 1$. Figure~\ref{fig:beta01} shows the evolution of the agents' trajectories (left) and velocities (right) for $\beta = 0.1$. The velocity profiles, reported in the right panel, exhibit convergence toward a common value, indicating that consensus is naturally achieved without control. This behavior is consistent with the theoretical prediction that for $\beta \leq 1/2$ the interaction strength remains sufficiently high to synchronize the agents.
In contrast, Figure~\ref{fig:beta1} shows that, for $\beta = 1$, consensus does not naturally occur in the uncontrolled case: agent trajectories (first row, left panel) diverge and velocities (second row, left panel) remain separated over time. 

\subsubsection{Direct control on velocities}
\noindent We investigate how the Z-control mechanism affects the emergence of consensus in velocity and analyze the influence of various control parameters on the system’s convergence behavior.
 Our objective is to design a Z-controlled model
\[
\begin{cases}
\dot{x}_i = v_i, \\[6pt]
\dot{v}_i = \sum_{j=1}^{N} a_{ij}(x)(v_j - v_i) + u_i, 
\end{cases}
\quad i = 1, \dots, N.
\]
by defining controls $u_i(t)$ so that 
the mean of velocities $\bar v(t) = \bar v(0)= \bar v$ and,
$v_i(t)\to \bar v$ exponentially with  rate $\lambda$.

The direct Z-control requires 
$$
u_i= -\lambda (v_i\,-\, \bar v) - \sum_{j=1}^{N}a_{ij}(x)(v_j - v_i). \,
$$
Figure~\ref{fig:beta1} illustrates the impact of the direct Z-control on velocities for $K=1$ and $\beta=1$. The panels show the evolution of positions, velocities, control inputs, and the consensus parameter $\Gamma(t)$.
In the uncontrolled case, agent trajectories (left column, top panel) diverge significantly and velocity differences (left column, middle panel) persist over time, indicating the absence of natural consensus. Conversely, the application of Z-control (right column) aligns the velocities (middle panel) and stabilizes the positions, as confirmed by the rapid decay of $\Gamma(t)$ in the bottom-right panel. The control inputs (bottom-left panel) remain bounded and are primarily active during the initial transient phase, demonstrating the efficiency of the strategy in enforcing consensus even in regimes where $\beta \geq \frac{1}{2}$ would otherwise prevent alignment.

\begin{figure}[htbp]
    \centering
    \includegraphics[width=0.4\textwidth]{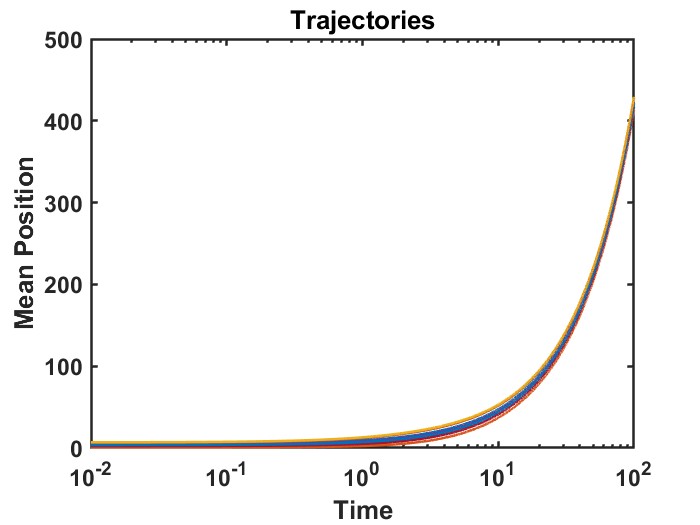}
 \includegraphics[width=0.4\textwidth]{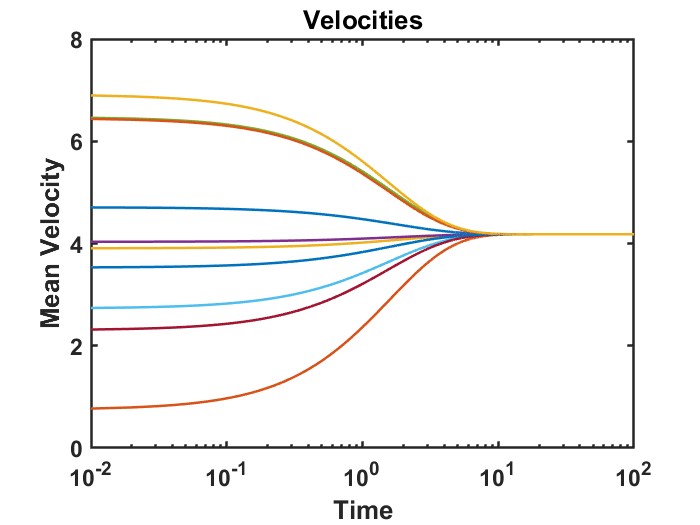}
  \caption{Cucker-Smale flocking model: uncontrolled trajectories (left) and velocities (right) for $\beta = 0.1$ are shown. Consensus is always achieved for values of the smoothness parameter $\beta \leq 1/2$.}
    \label{fig:beta01}
\end{figure}
\begin{figure}[htbp]
    \centering
\includegraphics[width=0.4\textwidth]{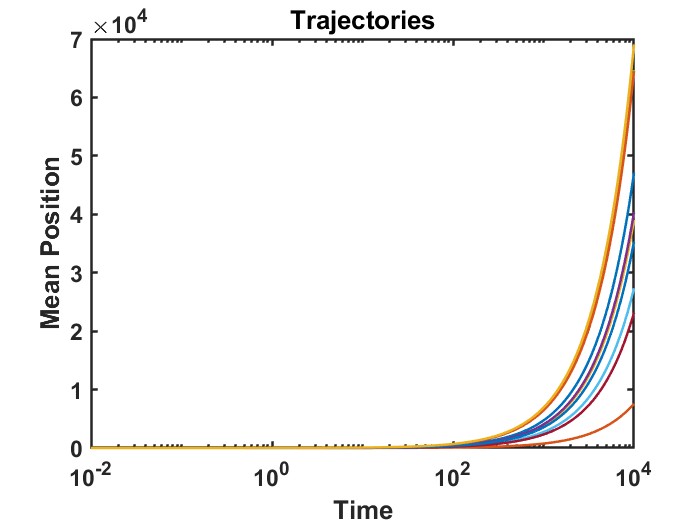}
\includegraphics[width=0.4\textwidth]{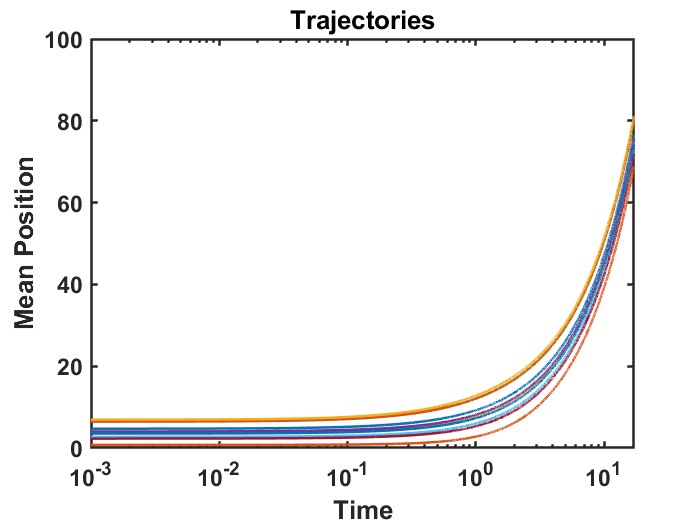}
\includegraphics[width=0.4\textwidth]{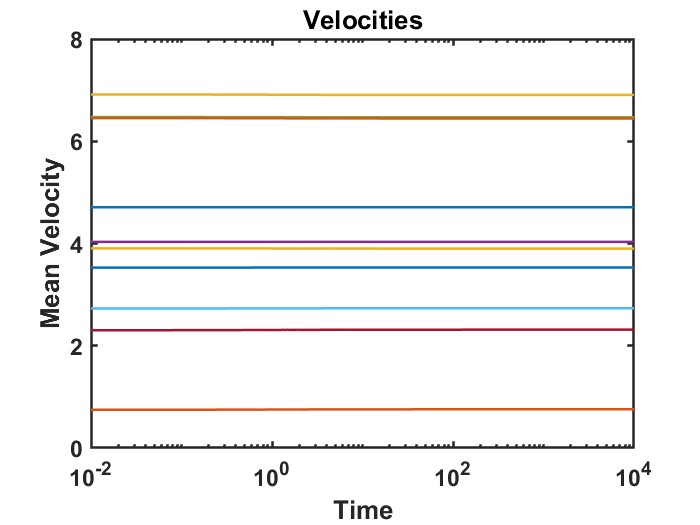}
\includegraphics[width=0.4\textwidth]{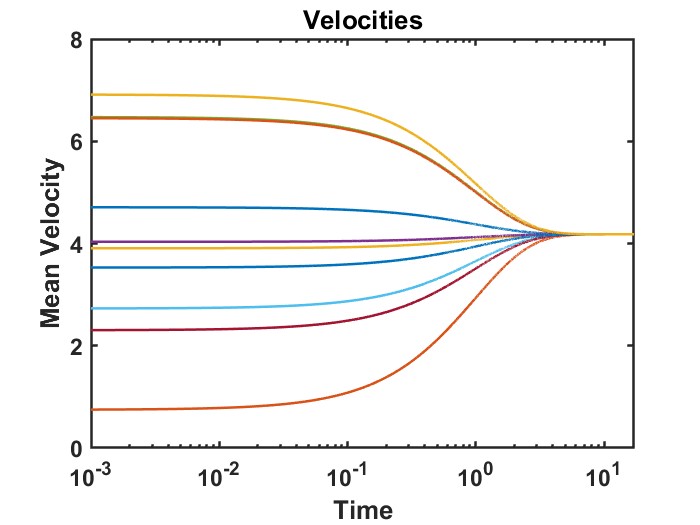}
\includegraphics[width=0.4\textwidth]{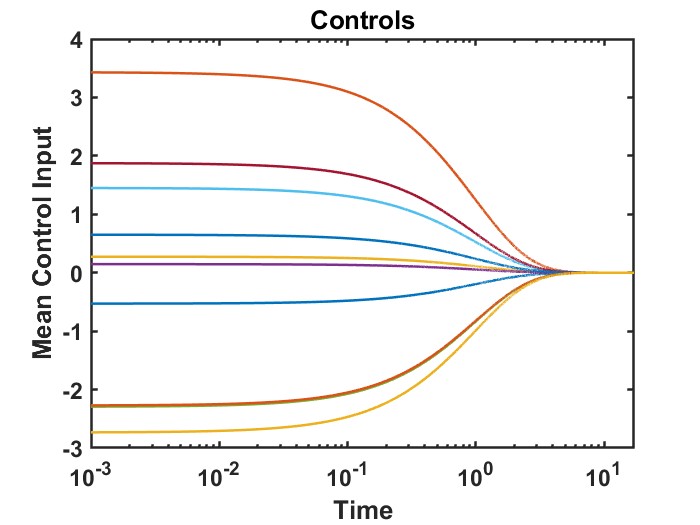}
\includegraphics[width=0.4\textwidth]{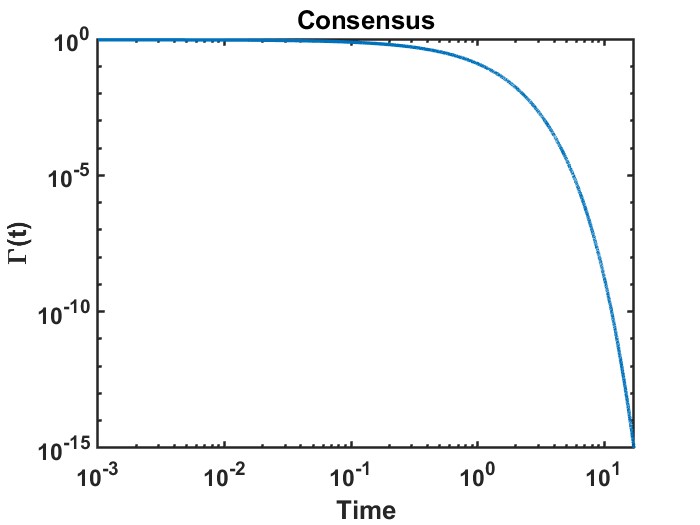}
\caption{Cucker-Smale flocking model. Uncontrolled and direct Z-controlled dynamics for $\beta = 1$ and $\lambda=1$. Top panels: uncontrolled (left) and controlled (right) agent mean value trajectories. Central panels: uncontrolled (left) and controlled (right) mean velocity evolution. Bottom panels: the control inputs (left) and the consensus parameter (right) over time. Initial conditions are selected such that consensus is not naturally achieved  ($\beta > 1/2$).}
    \label{fig:beta1}
\end{figure}

 
\subsubsection{Indirect control of velocities through positions}
\label{sec:2ind_pos}

Our objective is to design a Z-controlled model
\[
\begin{cases}
\dot{x}_i = v_i + u_i, \\[6pt]
\dot{v}_i = \sum\limits_{j=1}^{N} a_{ij}(x)(v_j - v_i),
\end{cases}
\quad i = 1, \dots, N,
\]
by defining suitable control inputs \( u_i(t) \) such 
that 
the velocities \( v_i(t) \) converge exponentially to their average $\bar{v}(t)$, with convergence rate \( \lambda > 0 \). 

To find the expression of $u_i$ first of all we use the second error formula for the error $e_i(t):=v_i(t)-\bar v(0)$:

$$
\Ddot e_i(t)\,+\, 2\, \lambda\, \dot e_i(t) +\lambda^2\, e_i(t)=0_d
$$

Since $\bar{v}(0)$  is constant, we have:
$
\dot{e}(t) = \dot{v}(t), \, \ddot{e}(t) = \ddot{v}(t)$
 and the error equation reads
 \begin{equation}\label{eq:error}
     \Ddot v_i(t)  +\, 2\, \lambda\, \dot v_i(t) +\lambda^2\, (v_i\,-\, \bar v(0))=0_d.
  \end{equation}

\begin{remark}
The error-design formula requires replacing the dynamics of $v_i$ with the exact solution of \eqref{eq:error}. 
The solution is
\[
v_i(t)=\bar v(0) 
+ e^{-\lambda t}\Big[\,v_i(0)-\bar v(0)
+ t\big(\dot v_i(0)+\lambda\,(v_i(0)-\bar v(0))\big)\Big].
\]
If, in addition, the initial derivative satisfies the consensus form
\[
\dot v_i(0)=\sum_{j=1}^{N} a_{ij}(x)\big(v_j(0)-v_i(0)\big),
\]
  then, due to weight-balance property of the matrix  \( A = (a_{ij}(x)) \),
\[
\frac{1}{N}\sum_{i=1}^{N} v_i(t)=\bar v(0)\quad\text{for all }t\ge 0,
\]
i.e., the average $\bar v(t)$ remains constant.
\end{remark}

Let us evaluate 
$$
\Ddot v_i(t)= \frac{d}{dt} \dot v_i(t)\, =\frac{d}{dt} \sum_{j=1}^{N} a_{ij}(x)(v_j - v_i) = \sum_{j=1}^{N} \frac{d}{dt}\Big( a_{ij}(x)(v_j - v_i)\Big)= \sum_{j=1}^{N} \dot  a_{ij}(x)(v_j - v_i) + a_{ij}(x)(\dot v_j - \dot v_i).
$$

Moreover,
\[
\begin{array}{rcl}
   \dot a_{ij}(x(t))   & =&   -\frac{2 \beta K}{N} \cdot \left(1 + \|x_i - x_j\|^2\right)^{-\beta - 1} \, (x_i - x_j)^\top(\dot{x}_i - \dot{x}_j) \\
     &=  &-\frac{2 \beta K}{N} \cdot \left(1 + \|x_i - x_j\|^2\right)^{-\beta - 1} \left[ (x_i - x_j)^\top(v_i - v_j) + (x_i - x_j)^\top(u_i - u_j) \right] \\
     &=&  b_{ij}(x) \, (x_i - x_j)^\top(v_i - v_j) +  b_{ij}(x)\, (x_i - x_j)^\top(u_i - u_j),
\end{array}
\]
where 
\begin{equation}\label{eq:bij}
b_{ij}(x):= -\dfrac{2 \beta K}{N} \cdot \left(1 + \|x_i - x_j\|^2\right)^{-\beta - 1}.
\end{equation}

Hence,
\[
\ddot v_i(t)=  \sum_{j=1}^{N} \left[
b_{ij}(x)\, (x_i - x_j)^\top(v_i - v_j)\, (v_j - v_i)
+ b_{ij}(x)\, (x_i - x_j)^\top(u_i - u_j)\, (v_j - v_i)
+ a_{ij}(x)(\dot v_j - \dot v_i)
\right].
\]

Moreover,
\[
a_{ij}(x)(\dot{v}_j - \dot{v}_i) = \sum_{k=1}^N a_{ij}(x) \left( a_{jk}(x) - a_{ik}(x) \right) v_k\, +\, a_{ij}(x) r_i(x) \, v_i \,-\,a_{ij}(x) r_j(x)\, v_j
\]
where 
\begin{equation}
    \label{eq:def_ri}
 r_i(x) =\sum_{k=1}^N a_{ik}(x), \quad
  r_j(x)=\sum_{k=1}^N a_{jk}(x).
 \end{equation}

Substituting everything into the error dynamics equation (\ref{eq:error}) we obtain
\begin{align*}
\sum_{j=1}^N b_{ij}(x)\,(v_j - v_i)(x_i - x_j)^\top (u_i - u_j)
&= -\Big[
\sum_{j=1}^N b_{ij}(x)\,(x_i - x_j)^\top (v_i - v_j)(v_j - v_i) \\
&+ \sum_{j=1}^N \sum_{k=1}^N a_{ij}(x)\big(a_{jk}(x) - a_{ik}(x)\big)\,v_k \\
&  \,+r_i^2(x)\, v_i \,-\,  \sum_{j=1}^N a_{ij}(x) \,r_j(x)\, v_j \\
&+ 2\lambda \sum_{j=1}^N a_{ij}(x)(v_j - v_i)
+ \lambda^2\,(v_i - \bar{v})
\Big].
\end{align*}


We define the block matrix \(\mathbb{L}_B =  \mathbb{L}_B(x,v) \in \mathbb{R}^{Nd \times Nd} \), 
structured as an \( N \times N \) matrix of blocks \( (\mathbb{L}_B)_{ij} \in \mathbb{R}^{d \times d} \), given by:
\[   (\mathbb{L}_B)_{ij} =
\begin{cases}
\sum\limits_{k \ne i} b_{ik}(x)\, (v_k - v_i)(x_i - x_k)^\top & \text{if } i = j \\[6pt]
- b_{ij}(x)\, (v_j - v_i)(x_i - x_j)^\top & \text{if } i \ne j.
\end{cases}
\]

Then, the matrix system becomes blockwise:
\[
\sum_{j=1}^N (\mathbb{L}_B)_{ij} \, u_j= - \mathcal{R}_i, \quad \text{for each } i = 1, \dots, N,
\]
where \( \mathcal{R}_i = \mathcal{R}_i(x,v;\lambda) \in \mathbb{R}^{d} \) is given by:
\begin{equation}\label{eq:R_iuno}
\begin{aligned}
\mathcal{R}_i
&= \sum_{j=1}^{N}b_{ij}(x)\,(x_i - x_j)^\top (v_i - v_j)\,(v_j - v_i)
   + \sum_{j,k=1}^{N}a_{ij}(x)\big(a_{jk}(x) - a_{ik}(x)\big)\,v_k \\
&\quad + \,r_i^2(x)\, v_i \,-\,  \sum_j a_{ij}(x) \,r_j(x)\, v_j + 2\lambda \sum_{j=1}^{N}a_{ij}(x)\,(v_j - v_i)
   + \lambda^2\,(v_i - \bar v).
\end{aligned}
\end{equation}

Equivalently, in matrix form
\begin{equation}\label{eq:syst1}
\mathbb{L}_B\,U = -\mathcal{R},\end{equation}
where $
U = \begin{bmatrix} u_1 \\ u_2 \\ \vdots \\ u_N \end{bmatrix} \in \mathbb{R}^{Nd}$ and  $\mathcal{R} = \begin{bmatrix} \mathcal{R}_1 \\ \mathcal{R}_2 \\ \vdots \\ \mathcal{R}_N \end{bmatrix} \in \mathbb{R}^{Nd}$.







We finally compute 
$U$ by the least-squares \eqref{eq:baseline-stacked} (see Section~\ref{sec:stacked-solve}). 

For $K=1$ and $\beta=1$ the result of the action of the indirect   Z-control on velocities through positions is shown in Figure~\ref{fig:beta1_indirect}.  The figure shows that the control successfully enforces alignment of velocities, as evidenced by the convergence of the mean velocities (top right panel) and the rapid decay of the consensus parameter $\Gamma(t)$ (bottom right panel). The control inputs (reported in the bottom left panel) exhibit a transient phase with relatively large magnitudes before vanishing, indicating that most of the effort is concentrated at the beginning of the process. This confirms the ability of the indirect strategy to achieve velocity consensus even in regimes where it would not naturally occur ($\beta \geq \frac{1}{2}$).

\begin{figure}[htbp]
    \centering
\includegraphics[width=0.4\textwidth]{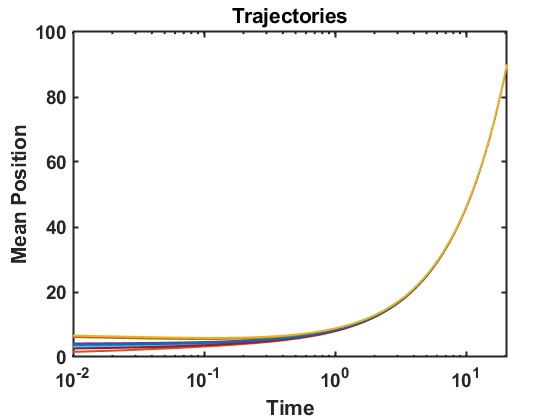}
\includegraphics[width=0.4\textwidth]{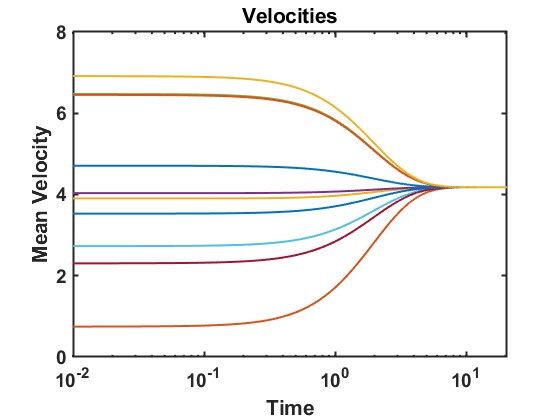}
\includegraphics[width=0.4\textwidth]{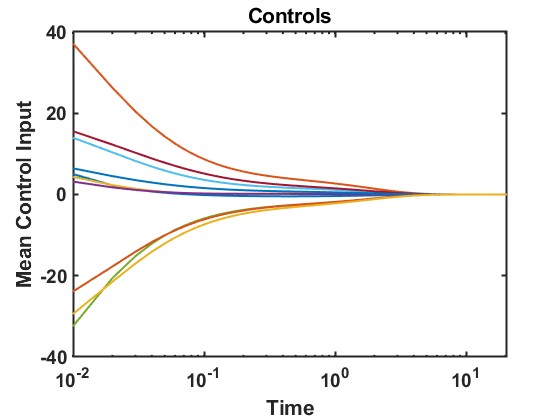}
\includegraphics[width=0.4\textwidth]{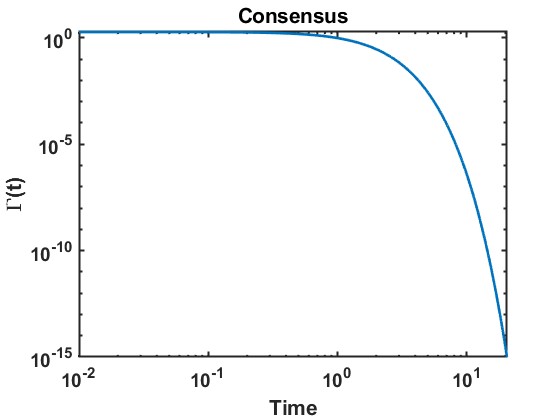}
\caption{Cucker-Smale flocking model. Indirect Z-controlled system dynamics for $\beta = 1$ and $\lambda=1$. The first row shows agent mean value trajectories and  mean velocity evolution, and the second row shows the control inputs and the consensus parameter over time. Initial conditions are selected such that consensus is not naturally achieved  ($\beta > 1/2$).}
    \label{fig:beta1_indirect}
\end{figure}

\subsection{Third-Order generalized Cucker-Smale Flocking Model}
The model is described by the following system of equations:
\begin{equation}\label{eq:CS3}
    \begin{cases}
    \dot{x}_i = v_i, \\[6pt]
    \dot{v}_i = z_i \\[6pt] \dot{z}_i = \sum_{j=1}^{N} a_{ij}(x)(z_j - z_i), 
    \end{cases}
    \quad a_{ij}(x) = \frac{K}{N(1 + \|x_i - x_j\|^2)^\beta}, \quad i,j  = 1, \dots, N, i\neq j,
\end{equation}
where \( K > 0 \) denotes the interaction strength and \( \beta \geq 0 \) determines the rate at which interaction decays with distance.

\subsubsection{Direct control on accelerations}
\noindent We investigate how the Z-control mechanism affects the emergence of consensus in accelerations.
 Our objective is to design a Z-controlled model
\[
\begin{cases}
\dot{x}_i = v_i, \\[6pt]
\dot{v}_i = z_i, \\[6pt] \dot{z}_i = \sum_{j=1}^{N} a_{ij}(x)(z_j - z_i) + u_i, 
\end{cases}
\quad i = 1, \dots, N,
\]
by defining controls $u_i(t)$ so that 
the mean of accelerations $\bar z(t) = \bar z(0)= \bar z$ and,
$z_i(t)\to \bar z$ exponentially with  rate $\lambda$.

Following the same steps as before 
$$
u_i= -\lambda (z_i\,-\, \bar z) - \sum_{j=1}^{N}a_{ij}(x)(z_j - z_i). \,
$$ 
Figure~\ref{fig:beta1_third} shows the effect of the direct Z-control on the third-order generalized Cucker-Smale system \eqref{eq:CS3} for $K =1$, $\beta =1$ and $\lambda=1$. The panels illustrate the evolution of positions, velocities, accelerations, control inputs, and the consensus parameter $\Gamma(t)$. In the uncontrolled case (left column), agent trajectories and velocities exhibit unbounded growth, and accelerations remain separated, indicating the absence of natural consensus. Conversely, under Z-control (right column), accelerations converge toward alignment, whereas positions and velocities stabilize relative to the uncontrolled case. This is confirmed by the rapid decay of $\Gamma(t)$ in the bottom-right panel. The control inputs (bottom-left panel) remain bounded and act primarily during the initial transient phase, demonstrating the effectiveness of the strategy in enforcing consensus even in higher-order dynamics where $\beta \geq \frac{1}{2}$ would otherwise prevent alignment.

\begin{figure}[htbp]
    \centering
\includegraphics[width=0.4\textwidth]{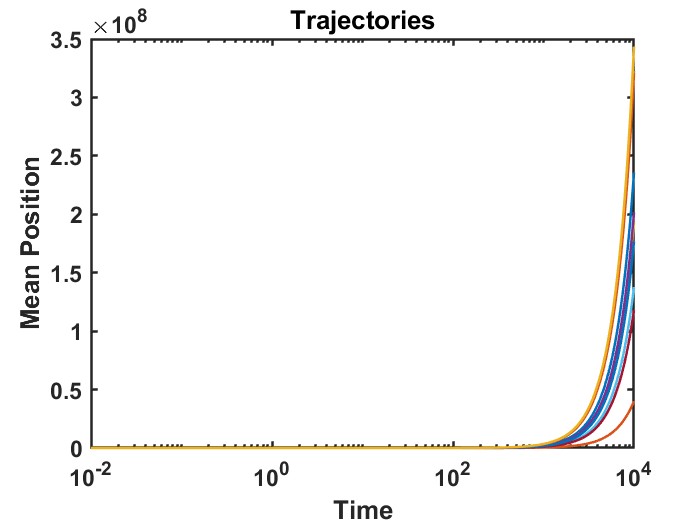}
\includegraphics[width=0.4\textwidth]{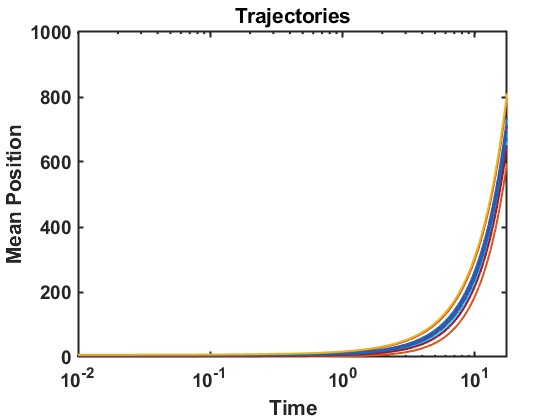}
\includegraphics[width=0.4\textwidth]{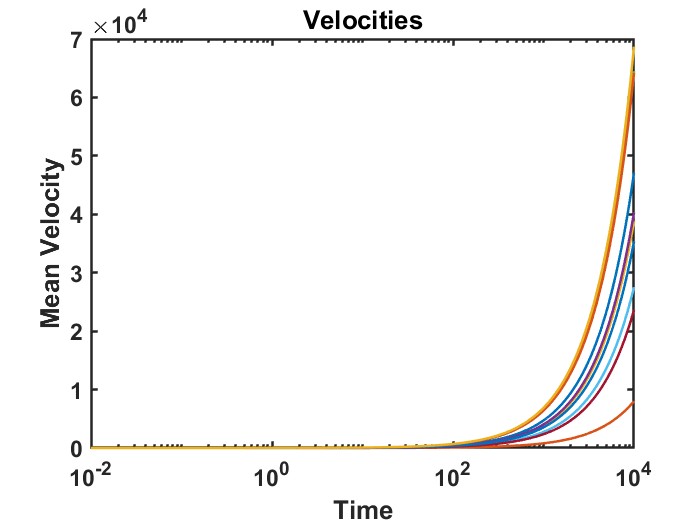}
\includegraphics[width=0.4\textwidth]{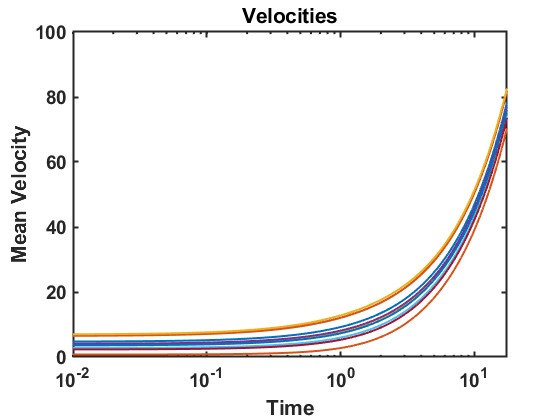}
\includegraphics[width=0.4\textwidth]{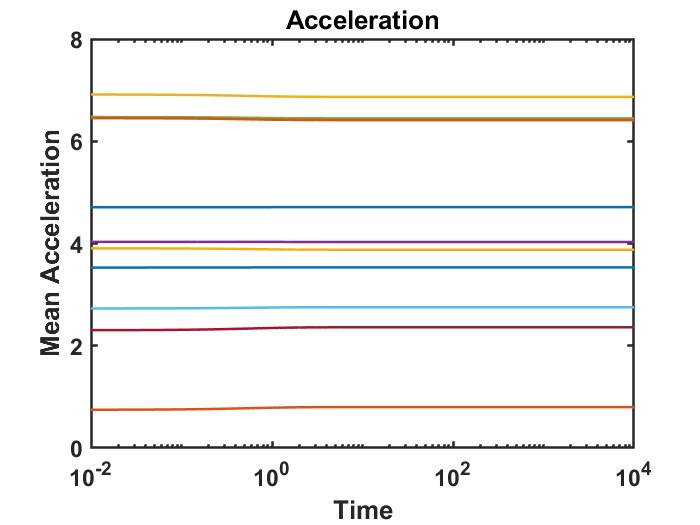}
\includegraphics[width=0.4\textwidth]{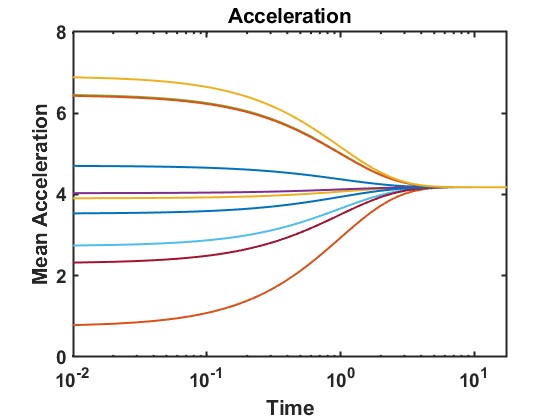}
\includegraphics[width=0.4\textwidth]{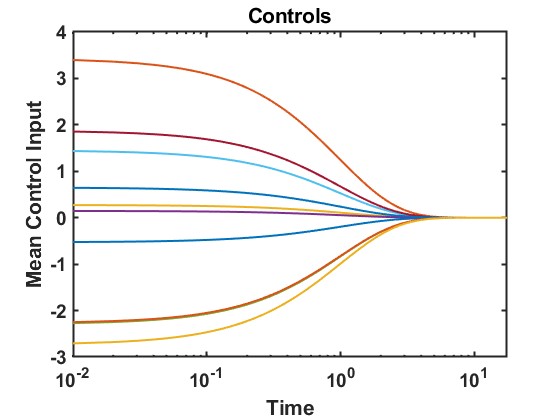}
\includegraphics[width=0.4\textwidth]{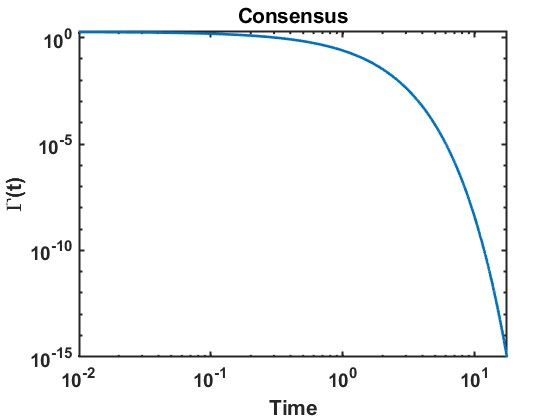}
\caption{Third-order generalized Cucker-Smale system. Uncontrolled and direct Z-controlled flocking dynamics for $\beta =1$ and $\lambda =1$. First row: uncontrolled (left) and controlled (right) agent mean value trajectories. Second row: uncontrolled (left) and
controlled (right) mean velocity evolution. Third row: uncontrolled (left) and
controlled (right) mean acceleration evolution. Bottom panels: the control inputs (left) and the consensus parameter (right)
over time. Initial conditions are selected such that consensus is not naturally achieved, highlighting the effectiveness of the Z-control strategy.}
    \label{fig:beta1_third}
\end{figure}

\subsubsection{Indirect control of accelerations through positions}
\label{sec:3D_ind_pos}
Our objective is to design a Z-controlled model
\[
\begin{cases}
\dot{x}_i = v_i + u_i \\[6pt] 
\dot{v}_i = z_i 
\\[6pt] 
\dot{z}_i = \sum\limits_{j=1}^{N} a_{ij}(x)(z_j - z_i),
\end{cases}
\quad i = 1, \dots, N,
\]
by defining suitable control inputs \( u_i(t) \) acting on $x_i$ such 
that the accelerations \( z_i(t) \) converge exponentially to their average $\bar{z}$, with convergence rate \( \lambda > 0 \). The average acceleration \( \bar{z}(t) \) remains constant throughout the evolution due to the symmetry of the matrix \( A = (a_{ij}(x)) \). 

As before, to find the expression of $u_i$, we need to use the second error formula for the error $e_i(t):=z_i(t)-\bar z(0)$:

$$
\Ddot e_i(t)\,+\, 2\, \lambda\, \dot e_i(t) +\lambda^2\, e_i(t)=0_d
$$
and  the error equation reads
 \begin{equation}\label{eq:error2}
     \Ddot z_i(t)  +\, 2\, \lambda\, \dot z_i(t) +\lambda^2\, (z_i\,-\, \bar z(0))=0_d.
  \end{equation}
  \begin{remark}
The error-design formula requires replacing the dynamics of $z_i$ with the exact solution of \eqref{eq:error2}. 
The solution is
\[
z_i(t)=\bar z(0) 
+ e^{-\lambda t}\Big[\,z_i(0)-\bar z(0)
+ t\big(\dot z_i(0)+\lambda\,(z_i(0)-\bar z(0))\big)\Big].
\]
If, in addition, the initial derivative satisfies the consensus form
\[
\dot z_i(0)=\sum_{j=1}^{N} a_{ij}(x)\big(z_j(0)-z_i(0)\big),
\]
  then, due to weight-balance property of the matrix  \( A = (a_{ij}(x)) \),
\[
\frac{1}{N}\sum_{i=1}^{N} z_i(t)=\bar z(0)\quad\text{for all }t\ge 0,
\]
i.e., the average $\bar z(t)$ remains constant.
\end{remark}

Let us evaluate 
$$
\Ddot z_i(t)=  \sum_{j=1}^{N} \dot  a_{ij}(x)(z_j - z_i) + a_{ij}(x)(\dot z_j - \dot z_i).
$$

With the same notation used  in \eqref{eq:bij}  
\[
\begin{array}{rcl}
   \dot a_{ij}(x)   & =& 
     b_{ij}(x) \, (x_i - x_j)^\top(v_i - v_j) + b_{ij}(x)\, (x_i - x_j)^\top(u_i - u_j)\, 
\end{array}
\]

and, 
\[
\ddot z_i(t)=  \sum_{j=1}^{N} \left[
b_{ij}(x)\, (x_i - x_j)^\top(v_i - v_j)\, (z_j - z_i)
+ b_{ij}(x)\, (x_i - x_j)^\top(u_i - u_j)\,( z_j - z_i)
+ a_{ij}(x)(\dot z_j - \dot z_i)
\right].
\]

Moreover, 
\[
a_{ij}(x)(\dot{z}_j - \dot{z}_i) = \sum_{k=1}^N a_{ij}(x) \left( a_{jk}(x) - a_{ik}(x) \right) z_k\, +\, a_{ij}(x) r_i(x) \, z_i \,-\,a_{ij}(x) r_j(x)\, z_j
\]
with $r_i$ and $r_j$ are defined in \eqref{eq:def_ri}.
Substituting everything into the error dynamics equation (\ref{eq:error2})
we obtain:
\begin{align*}
\sum_{j=1}^N b_{ij}(x)\,(z_j - z_i)(x_i - x_j)^\top (u_i - u_j)
&= -\Big[\sum_{j=1}^N b_{ij}(x)\,(x_i - x_j)^\top (v_i - v_j)(z_j - z_i)\\
&+ \sum_{j,k=1}^{N} a_{ij}(x)\big(a_{jk}(x) - a_{ik}(x)\big) z_k\\
& +r_i^2(x)\, z_i \,-\,  \sum_{j=1}^N a_{ij}(x) \,r_j(x)\, z_j \\
 & + 2\lambda \sum_{j=1}^N a_{ij}(x)(z_j - z_i)
  + \lambda^2\big(z_i - \bar z\big)\Big].
\end{align*}


With abuse of notations, we define the block matrix \(\mathbb{L}_B =  \mathbb{L}_B(x,z) \in \mathbb{R}^{Nd \times Nd} \), 
structured as an \( N \times N \) matrix of blocks \( (\mathbb{L}_B)_{ij} \in \mathbb{R}^{d \times d} \), given by:
\begin{equation}\label{eq:LB_block_zx}
    (\mathbb{L}_B)_{ij} =
\begin{cases}
\sum\limits_{k \ne i} b_{ik}(x)\, (z_k - z_i)(x_i - x_k)^\top & \text{if } i = j \\[6pt]
- b_{ij}(x)\, (z_j - z_i)(x_i - x_j)^\top & \text{if } i \ne j
\end{cases}
\end{equation}

Then, the matrix system becomes blockwise:
\[
\sum_{j=1}^N (\mathbb{L}_B)_{ij} \, u_j= - \mathcal{R}_i, \quad \text{for each } i = 1, \dots, N,
\]
where \(\mathcal{R}_i = \mathcal{R}_i(x,v,z;\lambda) \in \mathbb{R}^{d} \) is given by:
\begin{equation}\label{eq:R_idue}
\begin{aligned}
\mathcal{R}_i
&= \sum_{j=1}^{N}b_{ij}(x)\,(x_i - x_j)^\top (v_i - v_j)\,(z_j - z_i)
   + \sum_{j,k=1}^{N}a_{ij}(x)\big(a_{jk}(x) - a_{ik}(x)\big)\,z_k \\
&\quad  + \,r_i^2(x)\, z_i \,-\,  \sum_j a_{ij}(x) \,r_j(x)\, z_j + 2\lambda \sum_{j=1}^{N}a_{ij}(x)\,(z_j - z_i)
   + \lambda^2\,(z_i - \bar z).
\end{aligned}
    \end{equation}

In matrix form
\begin{equation}\label{eq:syst2}
\mathbb{L}_B\,U = -\mathcal{R},
\end{equation}
where $
U = \begin{bmatrix} u_1 \\ u_2 \\ \vdots \\ u_N \end{bmatrix} \in \mathbb{R}^{Nd}$ and  $\mathcal{R} = \begin{bmatrix} \mathcal{R}_1 \\ \mathcal{R}_2\\ \vdots \\ \mathcal{R}_N\end{bmatrix} \in \mathbb{R}^{Nd}$.

 Figure~\ref{fig:3D_indirect_pos} illustrates the behavior of the indirect Z-control on accelerations through positions, with $\beta = 1$, $K=1$ and $\lambda = 1$. The initial conditions are chosen so that consensus is not naturally achieved. The top row shows the evolution of agent trajectories and velocities, whereas the second row reports oscillatory behavior in control inputs (right), which do not decay to zero over time. This behavior is consistent with driving acceleration consensus through positions: the controlled variable must sustain nonvanishing inputs to enforce consensus at a higher derivative level. In contrast, when the target variable is controlled directly, or indirectly, via its immediate predecessor, control inputs vanish asymptotically.
 Notably, the acceleration profiles (bottom left) converge over time, indicating consensus on agent accelerations.
 
\begin{figure}[htbp]
    \centering
\includegraphics[width=0.4\textwidth]{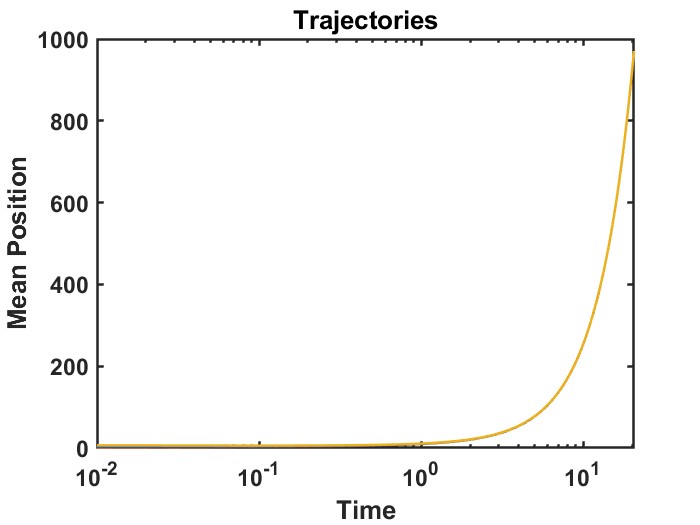}
\includegraphics[width=0.4\textwidth]{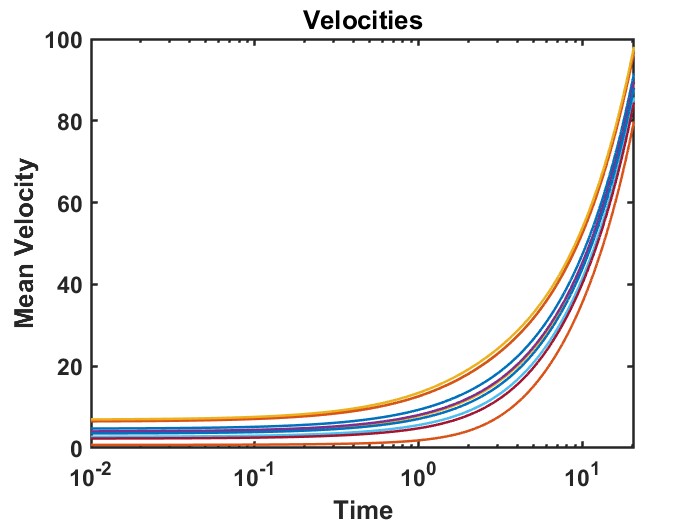}
\includegraphics[width=0.4\textwidth]{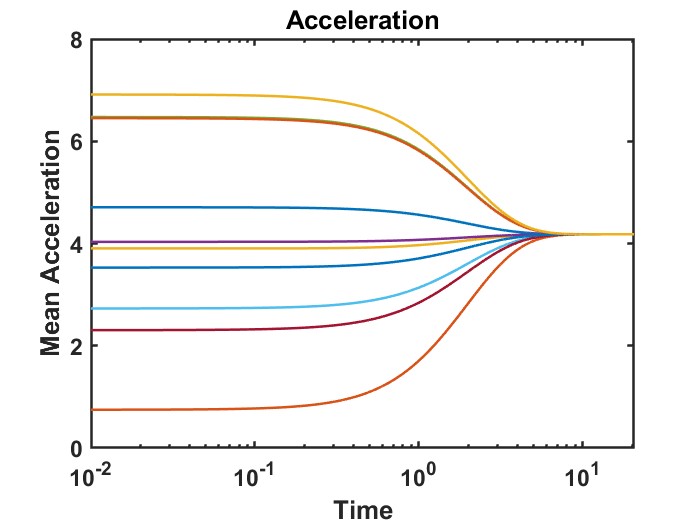}
\includegraphics[width=0.4\textwidth]{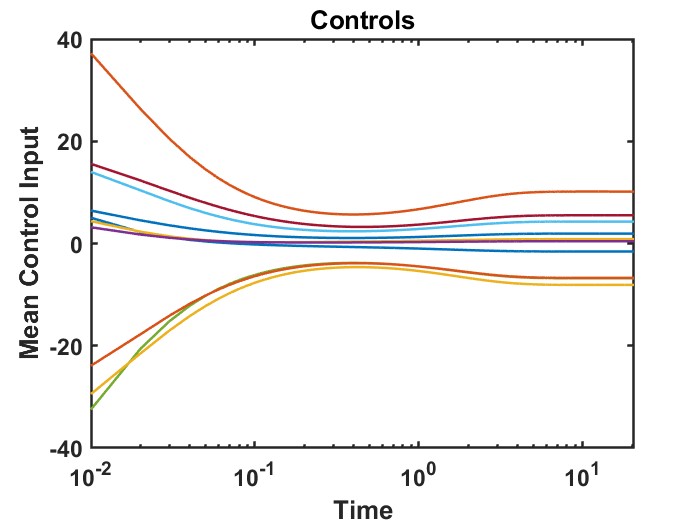}
\caption{Third-order generalized Cucker-Smale system. Indirect  Z-controlled system through positions for $\beta = 1$ and $\lambda=1$. The top panel shows agent mean value trajectories (left) and  mean velocity evolution (right), whereas the center panel shows the acceleration (left) and control inputs (right) over time. Initial conditions are selected such that consensus is not naturally achieved ($\beta > 1/2$). 
}
    \label{fig:3D_indirect_pos}
\end{figure}

\subsubsection{Indirect control of accelerations through velocities}\label{subsec:third-order-vel}

We consider
\[
\begin{cases}
\dot{x}_i = v_i, \\[3pt]
\dot{v}_i = z_i + u_i, \\[3pt]
\dot{z}_i = \sum\limits_{j=1}^{N} a_{ij}(x)\,(z_j - z_i),
\end{cases}
\qquad i=1,\dots,N,
\]
and design \(u_i(t)\) so that \(z_i(t)\) converges exponentially to the average \(\bar z\) with rate \(\lambda>0\), acting on the velocities. We impose
\begin{equation}\label{eq:third-order-z}
\dddot{z}_i + 3\lambda\,\ddot{z}_i + 3\lambda^2\,\dot{z}_i + \lambda^3\,(z_i-\bar z)=0_d.
\end{equation}

 For \(i\neq j\) set
\[
\rho_{ij}:=1+\|x_i-x_j\|^2,\quad
s_{ij}:=(x_i-x_j)^\top(v_i-v_j),\quad
q_{ij}:=\|v_i-v_j\|^2,\quad
r_{ij}:=(x_i-x_j)^\top(z_i-z_j),
\]
and define
\[
a_{ij}(x)=\dfrac{K}{N}\,\rho_{ij}^{-\beta},\qquad
b_{ij}(x)=-\dfrac{2\beta K}{N}\,\rho_{ij}^{-\beta-1}.
\]
Then,
\begin{equation}\label{eq:dot-aij}
\dot a_{ij}(x)= b_{ij}(x)\,(x_i-x_j)^\top(v_i-v_j)= b_{ij}\,s_{ij},
\end{equation}
and, using \(\dot v_i=z_i+u_i\),
\begin{equation}\label{eq:ddot-aij}
\begin{aligned}
\ddot a_{ij}(x)
&=\ddot a^{(0)}_{ij}(x)
\;+\; b_{ij}(x)\,(x_i-x_j)^\top(u_i-u_j).
\end{aligned}
\end{equation}
where $\ddot a^{(0)}_{ij}(x) =\dfrac{1}{N} \Big[4\beta(\beta+1)K\,\rho_{ij}^{-\beta-2}\,s_{ij}^2
-2\beta K\,\rho_{ij}^{-\beta-1}\,(q_{ij}+r_{ij})\Big]$.
Evaluate
\begin{equation}\label{eq:dotz}
\dot z_i = \sum_{j=1}^N a_{ij}(x)\,(z_j-z_i).
\end{equation}

Differentiating \eqref{eq:dotz} and using \eqref{eq:dot-aij},
\begin{equation}\label{eq:ddotz-pre}
\ddot z_i = \sum_{j=1}^N \Big( \dot a_{ij}(x)\,(z_j-z_i) + a_{ij}(x)\,(\dot z_j-\dot z_i) \Big).
\end{equation}
Since
\begin{equation*}
\dot z_j-\dot z_i = \sum_{k=1}^N \big(a_{jk}(x)-a_{ik}(x)\big)\,z_k-r_j(x)\,z_j+r_i(x)\,z_i,
\end{equation*}
where $r_i$ and $r_j$ defined in \eqref{eq:def_ri}, we have the fully expanded form
\begin{equation}\label{eq:ddotz}
\ddot z_i
= \sum_{j=1}^N \big(b_{ij}\,s_{ij}\big)\,(z_j-z_i)
\;+\; \sum_{j=1}^N\sum_{k=1}^N a_{ij}(x)\,\big(a_{jk}(x)-a_{ik}(x)\big)\,z_k + r_i^2(x)\,z_i - \sum_{j=1}^{N} a_{ij}(x)\,r_j(x)\,z_j.
\end{equation}
Differentiating \eqref{eq:ddotz-pre} once more and using \eqref{eq:ddot-aij}, \eqref{eq:dot-aij}, and
\begin{equation*}
\begin{aligned}
\ddot z_j-\ddot z_i
&= \sum_{k=1}^N \left(\dot{a}_{jk}(x)-\dot{a}_{ik}(x)\right)\,z_k + \sum_{k=1}^N \left(a_{jk}(x)-a_{ik}(x)\right)\,\dot{z}_k \\
&\quad -\dot{r}_{j}(x) \, z_j - r_j(x)\,\dot{z}_j +\dot{r}_{i}(x) \, z_i + r_i(x)\,\dot{z}_i,
\end{aligned}
\end{equation*}
we obtain
\begin{equation}
\label{eq:dddotz}
\begin{aligned}
\dddot z_i
&= \sum_{j=1}^N \ddot a_{ij}(x)\,(z_j-z_i)
\;+\; 2\sum_{j=1}^N \dot a_{ij}(x)\,(\dot z_j-\dot z_i)
\;+\; \sum_{j=1}^N a_{ij}(x)\,(\ddot z_j-\ddot z_i) \\[3pt]
&= \sum_{j=1}^N \Big(\ddot a^{(0)}_{ij}(x) + b_{ij}(x)\,(x_i-x_j)^\top(u_i-u_j)\Big)\,(z_j-z_i) + 2\sum_{j=1}^N\sum_{k=1}^N \big(b_{ij}\,s_{ij}\big)\,\big(a_{jk}(x)-a_{ik}(x)\big)\,z_k \\[2pt]
&\quad + 2\,r_i(x)\,\left(\sum_{j=1}^{N} b_{ij}\,s_{ij}\right) z_i -2\sum_{j=1}^{N} b_{ij}\,s_{ij}\,r_j(x)\,z_j(x) \\[2pt]
&\quad + \sum_{j,k=1}^{N} a_{ij}(x) \left(\dot{a}_{jk}(x)-\dot{a}_{ik}(x)\right) z_k + \sum_{j,k=1}^{N} \sum_{\ell=1}^N a_{ij}(x) \, a_{k\ell}(x) \left(a_{jk}(x)-a_{ik}(x)\right) \left(z_k-z_{\ell}\right) \\
&\quad - \sum_{j=1}^N a_{ij}(x) \, \dot{r}_j(x)\, z_j - \sum_{j,k=1}^{N}  a_{ij}(x)\,a_{jk}(x) \, r_j(x)\, \left(z_k-z_j\right) + r_i(x) \dot{r}_i(x) z_i + \sum_{j=1}^N a_{ij}(x)\, r_i^2(x) \, \left(z_j-z_i\right) .
\end{aligned}
\end{equation}
Replacing \(\dot a_{jk}\) in the last two lines by \(b_{jk}\,s_{jk}\) (as in \eqref{eq:dot-aij}) yields a version of \eqref{eq:dddotz} with no time derivatives left in the right-hand side. 

Substituting \eqref{eq:ddotz} and \eqref{eq:dddotz} into \eqref{eq:third-order-z} and moving all terms containing \(u\) to the left gives, for each \(i=1,\dots,N\),
\begin{equation}\label{eq:u-lhs}
\sum_{j=1}^N b_{ij}(x)\,(z_j-z_i)\,(x_i-x_j)^\top\,(u_i-u_j)
= -\,\mathcal{R}_i,
\end{equation}
where the right-hand side contains only known quantities; explicitly, $\mathcal{R}_i = \mathcal{R}_i(x,v,z;\lambda)$ given by
\begin{equation}\label{eq:R_itre}
\begin{aligned}
\mathcal{R}_i
&= \sum_{j=1}^{N}\ddot a^{(0)}_{ij}(x)\,(z_j-z_i)
 + 2\sum_{j,k=1}^{N}\big(b_{ij}\,s_{ij}\big)\,\big(a_{jk}(x)-a_{ik}(x)\big)\,z_k + 2\,r_i(x)\,\left(\sum_{j=1}^{N}b_{ij}\,s_{ij}\right) z_i  \\[2pt]
&\quad -2\sum_{j=1}^{N}b_{ij}\,s_{ij}\,r_j(x)\,z_j(x) + \sum_{j,k=1}^{N}a_{ij}(x) \left(\dot{a}_{jk}(x)-\dot{a}_{ik}(x)\right) z_k + \sum_{j,k,\ell =1}^N a_{ij}(x) \, a_{k\ell}(x) \left(a_{jk}(x)-a_{ik}(x)\right) \left(z_k-z_{\ell}\right) \\
&\quad - \sum_{j=1}^{N}a_{ij}(x) \, \dot{r}_j(x)\, z_j - \sum_{j,k=1}^{N} a_{ij}(x)\,a_{jk}(x) \, r_j(x)\, \left(z_k-z_j\right) + r_i(x) \dot{r}_i(x) z_i + \sum_{j=1}^{N}a_{ij}(x)\, r_i^2(x) \, \left(z_j-z_i\right) \\
&\quad + 3\lambda \sum_{j}^N \big(b_{ij}\,s_{ij}\big)\,(z_j-z_i)
\;+\; 3\lambda \sum_{j,k}a_{ij}(x)\,\big(a_{jk}(x)-a_{ik}(x)\big)\,z_k + 3\lambda\,r_i^2(x)\,z_i - 3\lambda \sum_{j=1}^{N}a_{ij}(x)\,r_j(x)\,z_j \\[2pt]
&\quad + 3\lambda^2\sum_{j=1}^N a_{ij}(x)\,(z_j-z_i) + \lambda^3\,(z_i-\bar z),
\end{aligned}
\end{equation}

with \(\ddot a^{(0)}_{ij}\) and \(b_{ij}\) given by \eqref{eq:ddot-aij} and \eqref{eq:dot-aij}.

 Defining \(\mathbb{L}_B(x,z)\in\mathbb{R}^{Nd\times Nd}\) by blocks 
$(\mathbb{L}_B)_{ij}$ given in \eqref{eq:LB_block_zx}, 
then, equation \eqref{eq:u-lhs} is equivalent to
\[
\sum_{j=1}^N (\mathbb{L}_B)_{ij}\,u_j \;=\; -\,\mathcal{R}_i,
\qquad i=1,\dots,N.
\]
In matrix form
\begin{equation}\label{eq:syst3}
    \mathbb{L}_B(x,z)\,U = -\mathcal{R},
    \end{equation}
where $
U = \begin{bmatrix} u_1 \\ u_2 \\ \vdots \\ u_N \end{bmatrix} \in \mathbb{R}^{Nd}$ and  $\mathcal{R} = \begin{bmatrix} \mathcal{R}_1 \\ \mathcal{R}_2\\ \vdots \\ \mathcal{R}_N\end{bmatrix} \in \mathbb{R}^{Nd}$.

\begin{figure}[t]
    \centering
    \includegraphics[width=0.4\textwidth]{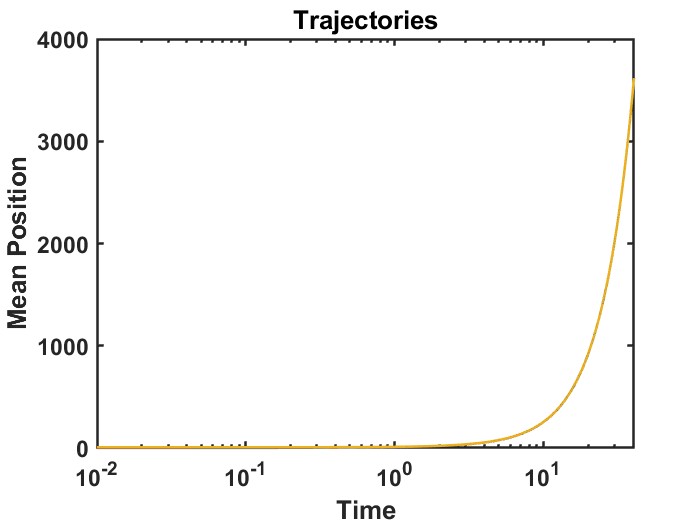}
\includegraphics[width=0.4\textwidth]{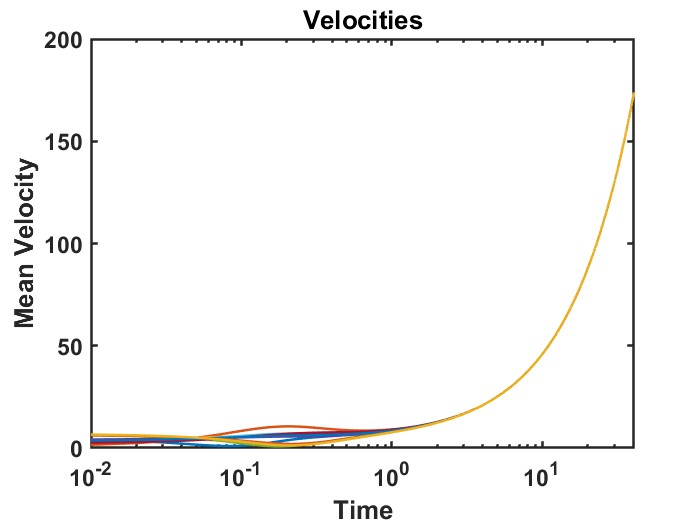}
\includegraphics[width=0.4\textwidth]{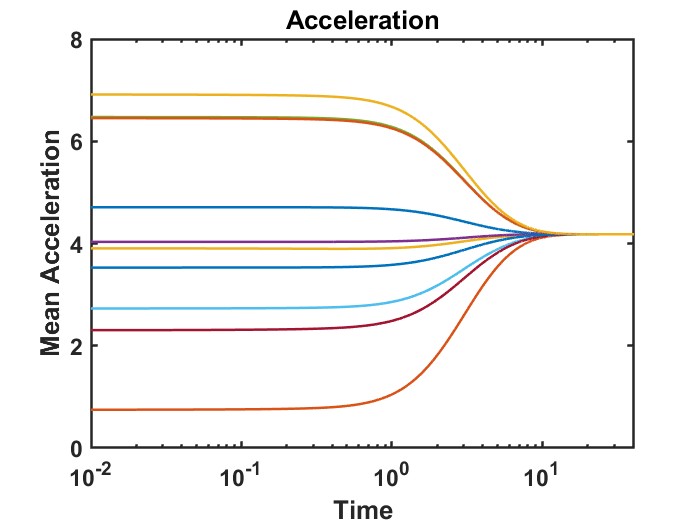}
\includegraphics[width=0.4\textwidth]{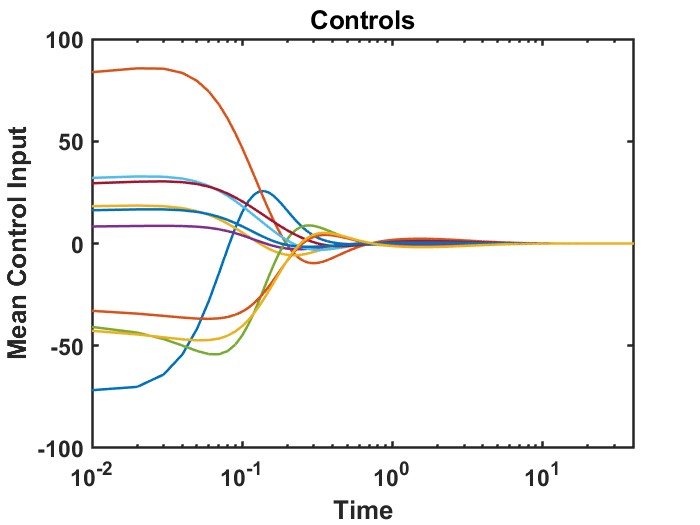}
  \caption{Third-order generalized Cucker-Smale system. Indirect Z-controlled dynamics through velocities for $\beta = 1$ and $\lambda=1$. The first row shows agent mean value trajectories (left) and mean velocity (right). The second rows reports accelerations evolution (left) and control inputs over time (right). Initial conditions are selected such that consensus is not naturally achieved. 
  }
    \label{fig:3D_indirect_vel}
\end{figure}

For $K=1$, $\beta=1$ and $\lambda =1$, the result of the action of the indirect Z-control on accelerations through velocities is shown in 
Figure~\ref{fig:3D_indirect_vel}. As in the previous case, initial conditions are selected so that consensus is not naturally achieved. The top row shows the evolution of agent trajectories (left) and velocities (right), while the second row reports acceleration profiles (left) and control inputs (right). Compared to the indirect Z-control through positions in Figure~\ref{fig:3D_indirect_pos}, the control inputs here exhibit larger transient oscillations but decay to zero over time. This is consistent with enforcing acceleration consensus through velocities: despite the higher oscillatory transient, the velocity-based strategy yields vanishing control asymptotically, unlike the control through positions case where inputs remain persistent.
\section{Unified least-squares solve for indirect Z-control}
\label{sec:stacked-solve}

In all indirect Z-control cases 
we arrive at a linear system of the form
\begin{equation}
\label{eq:LB-linear}
\mathbb{L}_B U = - \mathcal{R},
\end{equation}
 where $
 U = \begin{bmatrix} u_1 \\ u_2 \\ \vdots \\ u_N \end{bmatrix} \in \mathbb{R}^{Nd}$ and  $\mathcal{R} = \begin{bmatrix} \mathcal{R}_1 \\ \mathcal{R}_2\\ \vdots \\ \mathcal{R}_N\end{bmatrix} \in \mathbb{R}^{Nd}$.  $\mathbb{L}_B \in \mathbb{R}^{Nd \times Nd}$ is block-structured, with
$(\mathbb{L}_B)_{ij} \in \mathbb{R}^{d \times d}$, and the sums of the blocks
along each block row and each block column vanish, i.e.,
\[
\sum_{j=1}^N (\mathbb{L}_B)_{ij} = 
\sum_{j=1}^N (\mathbb{L}_B)_{ji} =  0_{d\times d}.
\]

Before addressing how to compute $U$, we first clarify when the linear system \eqref{eq:LB-linear}
is solvable. As observed above, the block structure of $\mathbb{L}_B$ implies that it is singular: in particular, all block-row sums vanish, so $\mathbb{L}_B$ has nontrivial left and right kernels. Consequently, for a given right-hand side $\mathcal{R}$, system \eqref{eq:LB-linear} is solvable only if $\mathcal{R}$ lies in the range of $\mathbb{L}_B$, or equi\-valently if it is orthogonal to $\ker(\mathbb{L}_B^\top)$. In our setting, the latter always contains the  directions $1_N \otimes w$, $w \in \mathbb{R}^d$, and this immediately yields a  compatibility condition on $\mathcal{R}$:

\begin{equation}
\label{eq:zero-sum-necessary}
\sum_{i=1}^N \mathcal{R}_i = 0_d.
\end{equation}
Indeed, if $\mathbb{L}_B U = -\mathcal{R}$, then for any $w \in \mathbb{R}^d$,
\[
0 = (1_N \otimes w)^\top \mathbb{L}_B U = - (1_N \otimes w)^\top \mathcal{R}
= - \sum_{i=1}^N \langle \mathcal{R}_i, w \rangle,
\]
which implies \eqref{eq:zero-sum-necessary}.

We now verify $\sum_{i=1}^N \mathcal{R}_i=0_{d}$ for each indirect control law. Throughout, $A(x)$ is the weight-balanced interaction matrix 
and $b_{ij}=b_{ji}$ are symmetric nonnegative weights .
We write $\bar v := \frac{1}{N}\sum_{i=1}^N v_i$ and $\bar z := \frac{1}{N}\sum_{i=1}^N z_i$.

\begin{theorem}[Indirect velocity control via positions ]\label{thm:compat-via-x}
Consider  $\mathcal{R}_i=\mathcal{R}_i(x,v,\lambda)$ in \eqref{eq:R_iuno}.
Then $\sum_{i=1}^N \mathcal{R}_i(x,v,\lambda)=0_d$.
\end{theorem}

\begin{proof}
We prove that the sum over $i$ of each line in \eqref{eq:R_iuno} vanishes.

Consider
\[
S_1:=\sum_{i,j=1}^N b_{ij}\,\big[(x_i-x_j)^\top (v_i-v_j)\big]\,(v_j-v_i).
\]
If we interchange the indices $(i,j)$ and $ (j,i)$ and use the symmetry $b_{ij}=b_{ji}$ together with $(v_j - v_i)=-(v_i - v_j)$, we see that each unordered pair $\{i,j\}$ contributes a vector and its opposite. Therefore all such contributions cancel in pairs and we obtain $S_1 = 0_d$.

Since $A$ is weight-balanced, it follows that 
\[
S_2:=\sum_{i,j,k=1}^N a_{ij}\,\big(a_{jk}-a_{ik}\big)\,v_k
=\sum_{k=1}^N\Big(\sum_{i,j=1}^N a_{ij}^N\,a_{jk}- \sum_{i,j=1}^N a_{ij}\,a_{ik}\Big)v_k
\]
and 
\[
S_4:=\sum_{i,j=1}^N a_{ij}\,(v_j-v_i)
=\sum_{i,j=1}^N a_{ij}\,v_j - \sum_{i,j=1}^N a_{ij}\,v_i
\]
nullify. Moreover,
\begin{align*}
    S_3 :&= \sum_{i=1}^N \Big(r_i^2\, v_i - \sum_{j=1}^N a_{ij}\,  r_j\,v_j\Big) = \sum_{i=1}^N r_i^2\, v_i - \sum_{j=1}^N  r_j \, \sum_{i=1}^N a_{ij} \,v_j
     \\
    &=
\sum_{i=1}^N r_i^2\, v_i -\sum_{j=1}^N  r_j \, \sum_{i=1}^N a_{ji}\,v_j\, = \,\sum_{i=1}^N r_i^2\, v_i -\sum_{j=1}^N  r_j \, r_j \,v_j = 0_d.
\end{align*}
Finally, $\sum_i (v_i-\bar v)=N\bar v - N\bar v=0_d$.
Adding all the terms gives $\sum_i \mathcal{R}_i=0_d$.
\end{proof}

\begin{theorem}[Indirect acceleration control via positions]\label{thm:compat-acc-via-x} Let  $\mathcal{R}_i=\mathcal{R}_i(x,v,z,\lambda)$ in \eqref{eq:R_idue}.
Then $\sum_{i=1}^N \mathcal{R}_i(x,v,z,\lambda)=0_d$.
\end{theorem}

\begin{proof}
The proof follows the same steps as Theorem~\ref{thm:compat-via-x}, replacing $v$ with $z$ and applying the same cancellations criteria.
\end{proof}

\begin{theorem}[Indirect  control via velocities ]\label{thm:compat-acc-via-v}
Let us consider 
$\mathcal{R}_i$ defined in \eqref{eq:R_itre}.
Then, $\sum_{i=1}^N \mathcal{R}_i(x,v,z,\lambda)=0_d$.
\end{theorem}

\begin{proof}
We prove that the sum over $i$ of  $\mathcal{R}_i$ is a finite sum of terms of a given type. We first we prove that  vanish those of the following structural types:
\begin{align*}
&S_1:=\sum_{i,j=1}^N \ c_{ij}(x,v,z,\lambda)\,(\xi_j-\xi_i),\qquad  c_{ij}=c_{ij}(x,v,z,\lambda), \, c_{ij}=c_{ji},\ \ \xi\in\{z,v\},\\
&S_2:=\sum_{i,j,k=1}^N  a_{ij}\,\big(a_{jk}(x)-a_{ik}\big)\,q_k,\qquad q\in\{z,v\}\ \text{(and time-derivatives, if present).}
\end{align*}
By symmetry $c_{ij}=c_{ji}$ and antisymmetry of differences, reindexing $(i,j)$ with $(j,i)$ shows that each unordered pair $\{i,j\}$ contributes a term and its negative. Hence, $S_1=0$. 

\noindent The $S_2$ term
 \[
S_2 = \sum_{i=1}^N \sum_{j,k=1}^N a_{ij}(a_{jk} - a_{ik}) \,q_k =\sum_{k=1}^N\Big(\sum_{i,j=1}^N a_{ij}a_{jk}-\sum_{i,j=1}^N a_{ij}a_{ik}\Big)q_k=0_d \qquad q \in \{z,v\} 
\]
and the  $S_3$ term
\begin{align*}
S_3 :&= \sum_{i=1}^N \sum_{j,k,\ell =1}^N a_{ij} \, a_{k\ell}\left(a_{jk}-a_{ik}\right) \left(z_k-z_{\ell}\right)\\
&= \sum_{k,\ell =1}^N a_{k\ell} \Big(\sum_{i,j}a_{ij} \,\left(a_{jk}-a_{ik}\right) \Big) \left(z_k-z_{\ell}\right) = 0_d 
\end{align*}
both vanish because of the weight-balanced property of $A(x)$. The row/column balance terms $S_4, S_5, S_6, S_7$
\begin{align*}
    &S_4 :=\sum_{i=1}^N \Bigg(r_i \sum_{j=1}^N b_{ij}\,s_{ij}\,z_i - \sum_{j=1}^N b_{ij}\,s_{ij}\,r_j\,z_j\Bigg), \quad
    S_5 :=\sum_{i=1}^N \Bigg(r_i\, \dot{r}_i\,z_i - \sum_{j=1}^N a_{ij}\,\dot{r}_j\,z_j\Bigg) \\
    &S_6:= \sum_{i=1}^N \left( \sum_{j=1}^N a_{ij} r_i^2 (z_j-z_i)-\sum_{j,k=1}^N a_{ij} a_{jk} r_j(z_k-z_j)\right), \\
    & 
   S_7:=\sum_{i=1}^N \Bigg(r_i^2 z_i - \sum_{j=1}^N a_{ij}\,r_j\,z_j\Bigg)
\end{align*}
cancel in pairs. Indeed, 
\begin{align*}
    S_4 &= \sum_{i=1}^N \Bigg(r_i \sum_{j=1}^N b_{ij}\,s_{ij} z_i - \sum_{j=1}^N b_{ij}\,s_{ij}\,r_j\,z_j\Bigg) \,  = \sum_{i=1}^N \Bigg(r_i \sum_{j=1}^N \dot{a}_{ij}\, z_i - \sum_{j=1}^N \dot{a}_{ij}\,r_j\,z_j\Bigg)\\
    &= \sum_{i=1}^N r_i \, \dot{r}_i\,z_i - \sum_{j=1}^N r_j\,z_j \sum_{i=1}^N \dot{a}_{ij}\, = \sum_{i=1}^N r_i \, \dot{r}_i\,z_i - \sum_{j=1}^N r_j\,z_j \sum_{i=1}^N \dot{a}_{ji} \\
    &= \sum_{i=1}^N r_i \, \dot{r}_i\,z_i - \sum_{j=1}^N r_j\,z_j \, \dot{r}_j = 0_d.
\end{align*}
The remaining cases $S_5=S_6=S_7=0$ follow by the same row/column balance arguments. Finally, $\sum_{i=1}^N (z_i-\bar z)=N\bar{z}-N\bar{z}=0_d$. Adding all the terms  yields $\sum_{i=1}^N \mathcal{R}_i=0_d$.
\end{proof}
\paragraph{Minimum-norm solution}\label{sec:5.3}In general, the  structure of $\mathbb{L}_B$, with rank-one off-diagonal
$d\times d$ blocks and vanishing block-row sums, together with the underlying
symmetries, reduces its effective rank, which we numerically observe to be
\begin{equation}\label{rank}
    \operatorname{rank}(\mathbb{L}_B)=Nd-\tfrac{d(d+1)}{2}.
\end{equation}
(see Table \ref{tab:Nd_rank_lambda}). This implies that $\ker(\mathbb{L}_B^\top)$ is strictly larger than
$\mathrm{span}\{1_N \otimes w\}$, so the compa\-tibility condition in
\eqref{eq:zero-sum-necessary} is necessary but not sufficient for solvability
of $\mathbb{L}_B U=-\mathcal{R}$. For this reason, we adopt a least-squares
formulation and select the control as the unique minimizer of
\begin{equation}\label{eq:baseline-stacked}
\min_{U\in\mathbb{R}^{Nd}}\;\bigl\|\mathbb{L}_B U+\mathcal{R}\bigr\|_2^2.
\end{equation}
Observe that the rank deficiency is
\(
Nd-\operatorname{rank}(\mathbb{L}_B)
= \frac{d(d+1)}{2},
\)
which depends only on $d$ and not on $N$. Thus, increasing the dimension $d$
decrease the rank; consequently,  in order to have a non-trivial minimum-norm solution, the rank of
$\mathbb{L}_B$ must be at least $1$. From (\ref{rank}) we obtain the constraint
$
Nd-\frac{d(d+1)}{2} \;\geq\; 1,
\label{eq:rank_geq_1}
$
and therefore, 
\[
N \ge \left\lceil \frac{d+1}{2} + \frac{1}{d} \right\rceil.
\]
In other words, the dimensions $N$ and $d$ cannot be chosen arbitrarily: for a
given $d$, the size $N$ must be sufficiently large so that the rank of $L_B$ is at least $1$ and the solution is
non-trivial.

\section{Higher-dimensional regimes and tuning of the parameter \texorpdfstring{$\lambda$}{lambda}}
\label{sec:highdim}
Our numerical experiments so far have deliberately targeted low to moderate sizes (small $d$ and small $N$) in order to isolate the structural effects of Z-control from pure scaling issues. We now investigate the performance of Z-control in higher-dimensional settings, considering both the direct and indirect control scenarios.
\paragraph{Direct Z-control}
Because direct Z-control acts directly on the top-order state, the algebraic complexity grows only mildly with $Nd$. Table~\ref{tab:consensus_times} reports the time needed to reach consensus under direct control for first, second, and third-order models with $N = 150$ and $d \in \{50, 100, 150\}$. As the table shows, increasing the state dimension has no appreciable impact on the consensus time. 
\begin{table}[htbp]
\small
\centering
\begin{tabular}{ccccc}
\toprule
& & \multicolumn{3}{c}{\textbf{Order}} \\
\cmidrule{3-5}
& & $\mathbf{1st}$ & $\mathbf{2nd}$ & $\mathbf{3rd}$ \\
\midrule
\multirow{3}{*}{\textbf{d}} 
& $\mathbf{50}$ & $81.059$ & $93.440$ & $94.799$ \\
& $\mathbf{100}$  & $85.375$ & $100.735$ & $101.534$ \\
& $\mathbf{150}$ & $88.752$ & $105.391$ & $106.258$ \\
\bottomrule\\
\end{tabular}
\caption{CPU time (s) to consensus for first, second, and third-order models with $N = 150$ agents and increasing state dimension $d$, under direct Z-control with $\lambda = 1$.}
\label{tab:consensus_times}
\end{table}
To further illustrate the robustness of direct Z-control in high-dimensional settings, Figure~\ref{fig:directHD} displays the decay of the consensus for $N=150$ agents of dimension $d=150$ across the three models and varying values of the control parameter $\lambda\in\{0.1,1,5,10\}$. The results confirm that direct control ensures rapid convergence independently on the order, and that increasing $\lambda$ accelerates consensus without compromi\-sing numerical stability. This behavior complements the findings in Table~\ref{tab:consensus_times}, reinforcing the scalability and effectiveness of direct Z-control even in large-dimensional systems.

\begin{figure}[htbp]
    \centering
\includegraphics[width=0.45\linewidth]{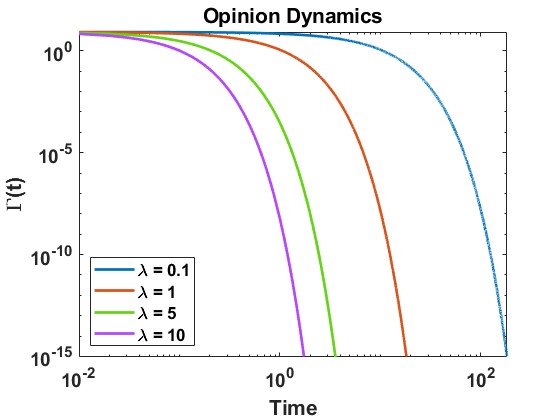}
\includegraphics[width=0.45\linewidth]{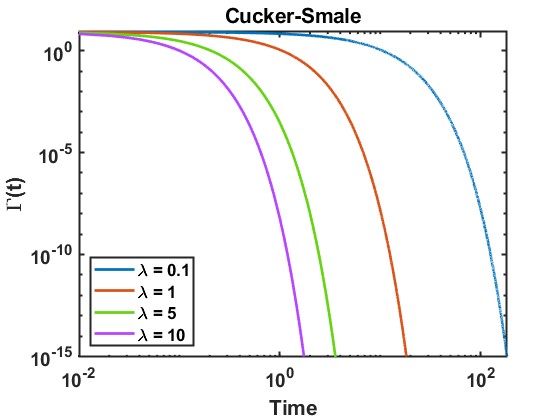}
\includegraphics[width=0.45\linewidth]{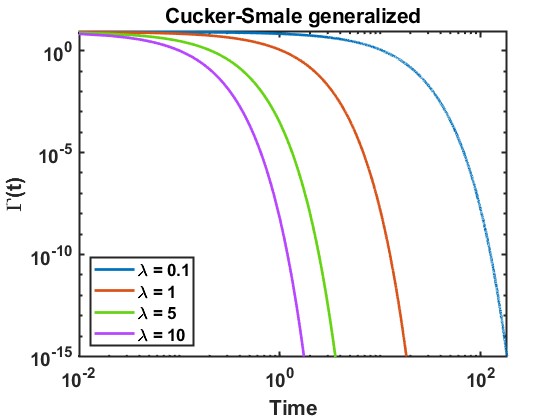}
    \caption{Time evolution of the consensus under direct Z-control for first, second, and third-order models with $N = 150$ and $d = 150$. Each panel shows the decay of $\Gamma(t)$ for different values of the control parameter $\lambda \in\{0.1,1,5,10\}$. Larger values of $\lambda$ lead to faster convergence, with all models exhibiting robust exponential decay.}
    \label{fig:directHD}
\end{figure}
\paragraph{Indirect Z-control}
In principle, in the indirect setting the design parameter $\lambda$ can still be chosen arbitrarily, but its value plays an important role in ensuring the numerical solvability of the system~\eqref{eq:baseline-stacked}. Our experiments show that when $\lambda$ is chosen too large, the largest singular value of the matrix $\mathbb{L}_B$ associated with the linear system collapses to zero. We therefore decrease $\lambda$ to slow down the convergence rate, so that the system approaches consensus more gradually and the numerical algorithm can better cope with ill-conditioning. For moderate and low dimensions, such as those used in our tests, the choice $\lambda = 1$ is sufficiently small to prevent numerical instabilities. For higher dimensions, however, we observed that much smaller values of $\lambda$ are required. In Table~\ref{tab:Nd_rank_lambda}, for fixed $N=150$, we report the maximum value $\lambda_{\max}$ that still guarantees convergence to consensus for the indirect Z-controlled Cucker-Smale model. However, as can be seen in Figure \ref{fig:ind_consensusHD}, reduced values of $\lambda$ imply a slower convergence to consensus. This together with the fact that increasing $d$ also increases the overall dimension of $\mathbb{L}_B$, this makes the algorithm impractical for high-dimensional models.
 \begin{table}[h!]
\centering
\small
\begin{tabular}{c c c c}
\hline 
$d$ & Dimension$(\mathbb{L}_B)$ & rank$(\mathbb{L}_B)$ & $\lambda_{max}$ \\
\hline
3  & 450  & 444  & 1  \\
4  & 600  & 590  & 0.7\\
5  & 750  & 735  & 0.6\\
10 & 1500 & 1445 & 0.3\\
20 & 3000 & 2790 & 0.1\\
30 & 4500 & 4035 & 0.05\\
\hline \\
\end{tabular}
\caption{Dimension and rank of $L_B$, and the corresponding threshold value $\lambda_{\max}$ for the indirect Z-controlled Cucker-Smale model. Results are reported for increasing state dimension $d$, with a fixed number of agents $N=150$.}
\label{tab:Nd_rank_lambda}
\end{table}
 A possible way to avoid reducing the rate parameter $\lambda$ is to rely on minimum-norm solutions obtained via Tikhonov regularization \cite{golub1999tikhonov}. However, this may come at the cost of deteriora\-ting convergence to consensus: since the solution is biased toward a zero-norm control, the trajectories can stagnate in a neighbourhood of consensus without actually reaching it, while the controls themselves decay to zero.

\begin{figure}
    \centering
\includegraphics[width=0.45\linewidth]{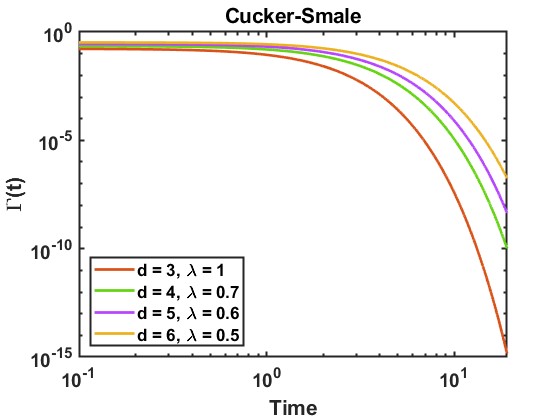}
\includegraphics[width=0.45\linewidth]{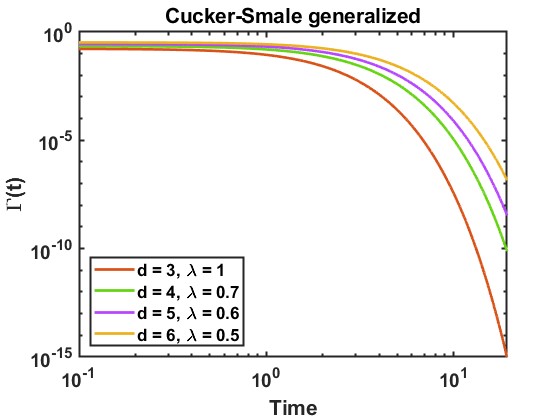}
    \caption{
Time evolution of the consensus $\Gamma(t)$  for the indirect Z-controlled Cucker-Smale model (left) and its generalized version (right) with $N=150$. Curves correspond to increasing dimensions $d$ and to the associated threshold values $\lambda_{\max}$ ensuring convergence to consensus. As $d$ grows, smaller $\lambda$ are required, resulting in a slower decay of $\Gamma(t)$ and hence a slower approach to consensus.}

    \label{fig:ind_consensusHD}
\end{figure}

MATLAB codes implementing the direct and indirect Z-control are available in the GitHub repository
\url{https://github.com/CnrIacBaGit/Z_control_multi_agents}.


 \section{Conclusions}\label{sec:conclusion}
We introduced a hierarchical Z-control methodology to stabilize consensus in multi-agent alignment models of arbitrary order. To the best of our knowledge, Z-control has not previously been explored for the class of alignment dynamics considered here; this work takes a first step in that direction.

In the \emph{direct} setting, Z-control is straightforward to implement and scales well with dimension, outperforming standard feedback laws for these dynamics \cite{monti2025hierarchicalclusteringdimensionalreduction}. Our numerical tests demonstrate that increasing the dimensions does not materially affect the consensus time.

Unlike classical approaches that assume actuation on the highest-order state, Z-control also supports \emph{indirect} actuation: consensus at the top-order variable is achieved by interve\-ning on lower-order dynamics (e.g., velocity or position). This capability is largely absent from prior formulations and offers a complementary, and arguably more practical, route to controlling collective behavior.
This indirect setting is common in environmental systems, where actuators affect only lower-order variables (e.g., flow fields or external cues) rather than the highest-order dynamics. In fish guidance near hydraulic structures (second-order models with velocity as the top-order state), operators steer velocity consensus by shaping ambient flow via gates, spillways, or jets that create preferred streamlines and reduce recirculation~\cite{nakai2025guidance, gautrais2009analyzing}. Direct top-order control would require agent-specific forces/torques and full-state feedback, which are unavailable. Similarly, for light-activated microswimmers, spatiotemporal light landscapes induce coherent motion changes and top-order consensus without direct actuation~\cite{lozano2016phototaxis}. Hence, indirect control is natural and practical, while direct highest-order actuation is often infeasible.

For indirect controllers, we derived a unified least-squares formulation and proved compatibility of the associated linear systems by showing that the right-hand sides in all models satisfy a zero-sum condition. However, this condition is necessary but not sufficient, as it does not by itself guarantee solvability of the system. We therefore adopt a least-squares formulation for determining the control solution. In this case, the dominant computational burden is the treatment of the $Nd \times Nd$ block matrix, which is intrinsically low rank and whose blocks collapse towards zero as consensus is approached. For low and moderate dimensions, such as those used in our tests, an appropriate choice of the parameter $\lambda$ is sufficient to avoid the numerical instabilities linked to the deterioration of the matrix rank. However, for higher dimensions we observed that very small values of $\lambda$ are necessary to reach consensus. This may result in computational times that are not acceptable, especially for real-world applications. The key algorithmic requirement is therefore a solver that remains accurate and efficient for matrices whose largest singular value decays to zero as consensus is reached, while fully exploiting the matrix's block structure and effective rank.

A promising way to mitigate the numerical challenges associated with the problem's structure and with high dimensionality is to couple Z-control with \emph{model order reduction} (MOR): project the network state and the control onto low-dimensional subspaces; solve the reduced least-squares problem; and lift the control back to the full space under suitable consistency constraints. Such  MOR ideas have proved effective in optimal control of agent systems (e.g., \cite{monti2025hierarchicalclusteringdimensionalreduction}) and are natural candidates for our least-squares solver. Moreover, the approach naturally extends to more complex agent behaviors, such as intermittency \cite{Monti_onoff,Diele_onoff} and spatial pattern formation \cite{alla2023adaptive,ALLApdmd,BOZZINI2021}, confirming its broader relevance. We plan to investigate MOR-enhanced Z-control in future work, aiming to preserve convergence to consensus while reducing the per-step computational cost.

\section*{Acknowledgements}

\noindent A.M. and F.D. research activity is funded by the National Recovery and Resilience Plan (NRRP), Mission 4 Component 2 Investment 1.4 - Call for tender No.
3138 of 16 December 2021, rectified by Decree n.3175 of 18 December 2021 of Italian Ministry of University and Research funded by the European Union - NextGenerationEU; Award Number: Project code CN 00000033, Concession Decree No. 1034 of 17 June 2022 adopted by the Italian Ministry of University and Research, CUP B83C22002930006, Project title \lq \lq National Biodiversity Future Centre''. A.M. acknowledges support from a scholarship (\emph{borsa per l’estero}) granted by Istituto Nazionale di Alta Matematica (INdAM) to carry out a research stay at Imperial College London. A.M. and F.D. are members of the INdAM research group GNCS. F.D. and A.M. would like to thank Mr. Cosimo Grippa for his valuable technical support.

\bibliographystyle{plain} 
\bibliography{biblio.bib}

@article{motsch2014heterophilious,
  title={Heterophilious dynamics enhances consensus},
  author={Motsch, S. and Tadmor, E.},
  journal={SIAM Review},
  volume={56},
  number={4},
  pages={577--621},
  year={2014},
  publisher={SIAM}
}

@article{CuckerSmale2007,
  author    = {Cucker, F. and Smale, S.},
  title     = {Emergent behaviour in flocks},
  journal   = {IEEE Transactions on Automatic Control},
  volume    = {52},
  number    = {5},
  pages     = {852--862},
  year      = {2007}
}

@incollection{CarrilloEtAl2010b,
  author    = {Carrillo, J. A. and Fornasier, M. and Toscani, G. and Vecil, F.},
  title     = {Particle, Kinetic, and Hydrodynamic Models of Swarming},
  booktitle = {Mathematical Modeling of Collective Behavior in Socio-Economic and Life Sciences},
  publisher = {Birkhäuser},
  year      = {2010}
}

@article{HaLiu2009,
  author    = {Ha, S.-Y. and Liu, J.-G.},
  title     = {A simple proof of the {C}ucker--{S}male flocking dynamics and mean-field limit},
  journal   = {Communications in Mathematical Sciences},
  volume    = {7},
  number    = {2},
  pages     = {297--325},
  year      = {2009}
}

@article{HaHaKim2010,
  author    = {Ha, S.-Y. and Ha, T. and Kim, J. H.},
  title     = {Emergent behavior of a {C}ucker--{S}male type particle model with nonlinear velocity couplings},
  journal   = {IEEE Transactions on Automatic Control},
  volume    = {55},
  number    = {7},
  pages     = {1679--1683},
  year      = {2010}
}

@inproceedings{ren2006high,
  title={High-order consensus algorithms in cooperative vehicle systems},
  author={Ren, Wei and Moore, Kevin and Chen, YangQuan},
  booktitle={2006 IEEE International Conference on Networking, Sensing and Control},
  pages={457--462},
  year={2006},
  organization={IEEE}
}

@article{he2011consensus,
  title={Consensus control for high-order multi-agent systems},
  author={He, Wangli and Cao, Jinde},
  journal={IET control theory \& applications},
  volume={5},
  number={1},
  pages={231--238},
  year={2011},
  publisher={IET}
}

@article{bailo2018optimal,
  title={Optimal consensus control of the {C}ucker-{S}male model},
  author={Bailo, R. and Bongini, M. and Carrillo, J. A and Kalise, D.},
  journal={IFAC-PapersOnLine},
  volume={51},
  number={13},
  pages={1--6},
  year={2018},
  publisher={Elsevier}
}

@article{lacitignola2021using,
  title={Using awareness to {Z}-control a {SEIR} model with overexposure: Insights on {C}ovid-19 pandemic},
  author={Lacitignola, D. and Diele, F.},
  journal={Chaos, Solitons \& Fractals},
  volume={150},
  pages={111063},
  year={2021},
  publisher={Elsevier}
}

@article{monti2025hierarchicalclusteringdimensionalreduction,
      title={Hierarchical clustering and dimensional reduction for optimal control of large-scale agent-based models}, 
      author={Monti, A. and Diele, F. and Kalise, D.},
      year={2025},
      journal={https://arxiv.org/abs/2507.19644},
      url={https://arxiv.org/abs/2507.19644}, 
}

@book{chung1997spectral,
  author = {Chung, F. R. K.},
  title = {Spectral Graph Theory},
  series = {CBMS Regional Conference Series in Mathematics},
  volume = {92},
  publisher = {AMS},
  year = {1997}
}

@article{xiao2004fast,
  author = {Xiao, L. and Boyd, S.},
  title = {Fast Linear Iterations for Distributed Averaging},
  journal = {Systems \& Control Letters},
  volume = {53},
  number = {1},
  pages = {65--78},
  year = {2004}
}

@article{olfati2004consensus,
  author = {Olfati-Saber, R. and Murray, R. M.},
  title = {Consensus Problems in Networks of Agents with Switching Topology and Time-Delays},
  journal = {IEEE Transactions on Automatic Control},
  volume = {49},
  number = {9},
  pages = {1520--1533},
  year = {2004}
}

@article{gharesifard2012ejc,
  author = {Gharesifard, B. and Cort{\'e}s, J.},
  title = {Distributed Strategies for Generating Weight-Balanced and Doubly Stochastic Digraphs},
  journal = {European Journal of Control},
  volume = {18},
  number = {6},
  pages = {539--557},
  year = {2012}
}

@article{GuoZhang2014,
  author  = {Guo, D. and Zhang, Y.},
  title   = {Neural Dynamics and {N}ewton--{R}aphson Iteration for Nonlinear Optimization},
  journal = {Journal of Computational and Nonlinear Dynamics},
  year    = {2014},
  volume  = {9},
  number  = {2},
  pages   = {021016}
}

@article{LiaoZhang2014,
  author  = {Liao, B. and Zhang, Y.},
  title   = {Different Complex {ZF}s Leading to Different Complex {ZNN} Models for Time-Varying Complex Generalized Inverse Matrices},
  journal = {IEEE Transactions on Neural Networks and Learning Systems},
  year    = {2014},
  volume  = {25},
  number  = {9},
  pages   = {1621--1631}
}

@article{ZhangLi2009,
  author  = {Zhang, Y. and Li, Z.},
  title   = {Zhang Neural Network for Online Solution of Time-Varying Convex Quadratic Programs Subject to Time-Varying Linear-Equality Constraints},
  journal = {Physics Letters A},
  year    = {2009},
  volume  = {373},
  pages   = {1639--1643},
  doi     = {10.1016/j.physleta.2009.03.004}
}

@book{ZhangYi2011,
  author    = {Zhang, Y. and Yi, C.},
  title     = {Zhang Neural Networks and Neural-Dynamic Method},
  year = {2011},
  isbn = {1616688394},
  publisher = {Nova Science Publishers, Inc.},
  address = {USA}
}

@article{ZhangEtAl2016,
  author  = {Zhang, Y. and Yan, X. and Liao, B. and Zhang, Y. and Ding, Y.},
  title   = {Z-type Control of Populations for {L}otka--{V}olterra Model with Exponential Convergence},
  journal = {Mathematical Biosciences},
  year    = {2016},
  volume  = {272},
  pages   = {15--23}
}

@article{MBS2016,
  author  = {Lacitignola, D. and Diele, F. and Marangi, C. and Provenzale, A.},
  title   = {On the Dynamics of a Generalized Predator--Prey System with {Z}-type Control},
  journal = {Mathematical Biosciences},
  year    = {2016},
  volume  = {280},
  pages   = {10--23}
}

@article{Samanta2018,
  author  = {Samanta, S.},
  title   = {Study of an Epidemic Model with {Z}-type Control},
  journal = {International Journal of Biomathematics},
  year    = {2018},
  volume  = {11},
  number  = {7},
  pages   = {1850084}
}

@article{SenapatiEtAl2020,
  author  = {Senapati, A. and Panday, P. and Samanta, S. and Chattopadhyay, J.},
  title   = {Disease Control Through Removal of Population Using {Z}-control Approach},
  journal = {Physica A},
  year    = {2020},
  volume  = {548},
  pages   = {123846}
}

@article{LacitignolaDiele2019,
  author  = {Lacitignola, D. and Diele, F.},
  title   = {On the {Z}-type Control of Backward Bifurcations in Epidemic Models},
  journal = {Mathematical Biosciences},
  year    = {2019},
  volume  = {315},
  pages   = {108215}
}

@article{AlzahraniEtAl2018,
  author  = {Alzahrani, A. K. and Alshomrani, A. S. and Pal, N. and Samanta, S.},
  title   = {Study of an {E}co-epidemiological {M}odel with {Z}-type {C}ontrol},
  journal = {Chaos, Solitons \& Fractals},
  year    = {2018},
  volume  = {113},
  pages   = {197--208}
}

@article{MandalEtAl2021,
  author  = {Mandal, D. and Chekroun, A. and Samanta, S. and Chattopadhyay, J.},
  title   = {A Mathematical Study of a Crop--Pest--Natural Enemy Model with {Z}-type Control},
  journal = {Mathematics and Computers in Simulation},
  volume = {187},
  pages = {468-488},
  year    = {2021},
  doi     = {10.1016/j.matcom.2021.03.014}
}

@article{nakai2025guidance,
  title={Guidance of fish to the fishway by gate release and confirmation of induction effect},
  author={Nakai, M. and Masumoto, T. and Pasternack, G. B and Asaeda, T.},
  journal={Journal of Ecohydraulics},
  volume={10},
  number={1},
  pages={58--78},
  year={2025},
  publisher={Taylor \& Francis}
}

@article{lozano2016phototaxis,
  title={Phototaxis of synthetic microswimmers in optical landscapes},
  author={Lozano, C. and Ten Hagen, B. and L{\"o}wen, H. and Bechinger, C.},
  journal={Nature communications},
  volume={7},
  number={1},
  pages={12828},
  year={2016},
  publisher={Nature Publishing Group UK London}
}

@book{ren2008distributed,
  title={Distributed consensus in multi-vehicle cooperative control: theory and applications},
  author={Ren, Wei and Beard, Randal W},
  year={2008},
  publisher={Springer}
}

@article{olfati2006flocking,
  title={Flocking for multi-agent dynamic systems: Algorithms and theory},
  author={Olfati-Saber, Reza},
  journal={IEEE Transactions on automatic control},
  volume={51},
  number={3},
  pages={401--420},
  year={2006},
  publisher={IEEE}
}

@article{vicsek1995novel,
  title={Novel type of phase transition in a system of self-driven particles},
  author={Vicsek, Tam{\'a}s and Czir{\'o}k, Andr{\'a}s and Ben-Jacob, Eshel and Cohen, Inon and Shochet, Ofer},
  journal={Physical review letters},
  volume={75},
  number={6},
  pages={1226},
  year={1995},
  publisher={APS}
}

@article{proskurnikov2018tutorial,
  title={A tutorial on modeling and analysis of dynamic social networks. Part II},
  author={Proskurnikov, Anton V and Tempo, Roberto},
  journal={Annual Reviews in Control},
  volume={45},
  pages={166--190},
  year={2018},
  publisher={Elsevier}
}

@article{nowzari2016analysis,
  title={Analysis and control of epidemics: A survey of spreading processes on complex networks},
  author={Nowzari, Cameron and Preciado, Victor M and Pappas, George J},
  journal={IEEE Control Systems Magazine},
  volume={36},
  number={1},
  pages={26--46},
  year={2016},
  publisher={IEEE}
}

@article{gautrais2009analyzing,
  title={Analyzing fish movement as a persistent turning walker},
  author={Gautrais, Jacques and Jost, Christian and Soria, Marc and Campo, Alexandre and Motsch, S{\'e}bastien and Fournier, Richard and Blanco, St{\'e}phane and Theraulaz, Guy},
  journal={Journal of mathematical biology},
  volume={58},
  number={3},
  pages={429--445},
  year={2009},
  publisher={Springer}
}

@article{golub1999tikhonov,
  title={Tikhonov regularization and total least squares},
  author={Golub, Gene H and Hansen, Per Christian and O'Leary, Dianne P},
  journal={SIAM journal on matrix analysis and applications},
  volume={21},
  number={1},
  pages={185--194},
  year={1999},
  publisher={SIAM}
}

@article{alla2023adaptive,
  title={Adaptive {POD}-{DEIM} correction for {T}uring pattern approximation in reaction--diffusion {PDE} systems},
  author={Alla, A. and Monti, A. and Sgura, I.},
  journal={Journal of Numerical Mathematics},
  volume={31},
  number={3},
  pages={205--229},
  year={2023},
  publisher={De Gruyter}
}

@article{Monti_onoff,
title = {On–off intermittency in population outbreaks: Reactive equilibria and propagation on networks},
journal = {Communications in Nonlinear Science and Numerical Simulation},
volume = {130},
pages = {107788},
year = {2024},
author = {A. Monti and F. Diele and C. Marangi and A. Provenzale}
}

@article{ALLApdmd,
title = {Piecewise {DMD} for oscillatory and Turing spatio-temporal dynamics},
journal = {Computers \& Mathematics with Applications},
volume = {160},
pages = {108-124},
year = {2024},
author = {A. Alla and A. Monti and I. Sgura}
}

@article{Diele_onoff,
author = {Diele, F. and Lacitignola, D. and Monti, A.},
title = {On–Off Intermittency and Long-Term Reactivity in a Host–Parasitoid Model with a Deterministic Driver},
journal = {International Journal of Bifurcation and Chaos},
volume = {34},
number = {02},
pages = {2450041},
year = {2024}
}

@article{BOZZINI2021,
title = {Model-reduction techniques for {PDE} models with {T}uring type electrochemical phase formation dynamics},
journal = {Applications in Engineering Science},
volume = {8},
pages = {100074},
year = {2021},
author = {B. Bozzini and A. Monti and I. Sgura}
}

@article{LMcircular2025,
    doi = {10.1371/journal.pone.0332348},
    author = {Lacitignola, Deborah AND Martiradonna, Angela},
    journal = {PLOS ONE},
    publisher = {Public Library of Science},
    title = {Can we enhance trust in the circular economy through referral marketing control?},
    year = {2025},
    month = {09},
    volume = {20},
    url = {https://doi.org/10.1371/journal.pone.0332348},
    pages = {1-32},
    number = {9}
}

\end{document}